\documentclass[11pt]{article}

\usepackage{amsmath,amsthm,amssymb,graphics}
\usepackage{mathrsfs}
\usepackage{fullpage}
\usepackage{amsfonts}
\usepackage{epsfig}
\usepackage{color}
\newtheorem{thm}{Theorem}[section]
\newtheorem{cor}[thm]{Corollary}
\newtheorem{prop}[thm]{Proposition}
\newtheorem{lem}[thm]{Lemma}
\theoremstyle{definition}
\newtheorem{defn}[thm]{Definition}

\newtheorem{rem}[thm]{Remark}
\theoremstyle{remark} \numberwithin{equation}{section}

\newcommand{\R}{\mathbb{R}}

\newcommand{\Z}{\mathbb{Z}}
\newcommand{\mathsym}[1]{{}}

\begin{document}
\setcounter{page}{1}

\title{\textbf{Existence and asymptotic properties for the solutions to nonlinear SFDEs driven by G-Brownian
motion with infinite delay}}
\author{\textbf{Faiz Faizullah$^{\footnote{E-mail:
faiz \b{} math@yahoo.com/faiz \b{}
math@ceme.nust.edu.pk/g.a.faizullah@swansea.ac.uk}}$}
 \vspace{0.1cm}\\
 {Department of Mathematics, Swansea University, Singleton Park SA2
8PP UK}
 \\{Department of BS and H, College of E and ME, National
University}\\{ of Sciences and Technology (NUST) Pakistan}}
\maketitle
\begin{abstract} The aim of this paper is to present the analysis for the
solutions of nonlinear stochastic functional differential equation
driven by G-Brownian motion with infinite delay (G-SFDEwID). Under
some useful assumptions, we have proved that the G-SFDEwID admits a
unique local solution. The mentioned theory has been further
generalized to show that G-SFDEwID admits a unique strong global
solution. The asymptotic properties, mean square boundedness and
convergence of solutions with different initial data have been
derived. We have assessed that the solution map $X_t$ is mean square
bounded and two solution maps from different initial data are
convergent. In addition, the exponential estimate for the solution
has been studied.

\end{abstract}

\section{Introduction}
The stochastic dynamical systems, in which the future state of the
systems not only relies on the current state but also on its past
history, lead to stochastic functional differential equations with
delays. These equations have tremendous applications in diverse
areas of sciences and engineering such as population dynamics
\cite{bm,myz}, epidemiology \cite{bg}, gene expression \cite{mpbf},
financial assets \cite{cy,t,z} and neural networks \cite{lf}. There
is by now a rather comprehensive mathematical literature on
existence, uniqueness, stability, moment estimates and other related
results of solutions for stochastic functional differential
equations \cite{m,m1,mc,m2,ms}. In the framework of G-Brownian
motion, the existence and uniqueness theorem for solutions to
stochastic functional differential equations with infinite delay has
been given by Ren, Bi and Sakthivel. Under the linear growth and
Lipschitz conditions, they have used the Picard approximation
technique \cite{rbs} while Faizullah has used the Cauchy-Maruyama
approximation scheme \cite{f6} to develop the mentioned theory. The
idea has been extended to non-Lipschitz conditions by Faizullah to
prove the existence-uniqueness theorem \cite{f2} and $p$th moment
estimates for the solutions to these equations \cite{f3,f4}.
Recently, Faizullah et. al, \cite{fm,fb} has generalized the theory
to determine existence, stability and the $p$th moment estimates for
solutions to neutral stochastic functional differential equations in
the G-framework (G-SFDEs). However, to the best of our knowledge, no
literature can be found on stochastic functional differential
equations driven by G-Brownian motion with infinite delay
(G-SFDEwID) in the phase space $C_q((-\infty,0];\R^d)$ defined
below. This article will contribute to fill the mentioned gap. By
using the truncation method, the global strong solutions for
G-SFDEwID will be explored. Furthermore, this article will present a
systematic study of the asymptotic properties for the solutions as
well as solution maps for G-SFDEwID. The $L^2_G$ and exponential
estimates will also be investigated. Let $\R^d$ and $A^\tau$ denote
$d$-dimensional Euclidean space and transpose of a matrix or vector
$A$ respectively. Let $C((-\infty,0];\R^d)$ be the collection of
continuous functions from $(-\infty,0]$ to $\R^d$, then for a given
number $q>0$ we define the phase space with the fading memory
$C_q((-\infty,0];\R^d)$ by
\begin{equation*} C_q((-\infty,0];\R^d)=\{\psi\in C((-\infty,0];\R^d):\lim_{\alpha\rightarrow -\infty}e^{q\alpha}\psi(\alpha)\, \text{exists in}\, \R^d
\}.
\end{equation*}
The space $C_q((-\infty,0];\R^d)$ is complete with norm
$\|\psi\|_q=\sup_{-\infty<\alpha\leq
0}e^{q\alpha}|\psi(\alpha)|<\infty$. This is a Banach space of
continuous and bounded functions and for any $0< q_1\leq
q_2<\infty$, $C_{q_1}\subseteq C_{q_2}$ \cite{ks,wym}. Let
$\mathcal{B}(C_q)$ be the $\sigma$-algebra generated by $C_q$ and
$C_q^0=\{\psi\in C_q:
\lim_{\alpha\rightarrow-\infty}e^{q\alpha}\psi(\alpha)=0\}$. Denote
by $L^2(C_q)$ (resp. $L^2(C^0_q)$) the space of all
$\mathcal{F}$-measurable $C_q$-valued (resp. $C^0_q$-valued)
stochastic processes $\psi$ such that $E\|\psi\|_q^2<\infty$. Let
$(\Omega,\mathcal{F},\mathbb{P})$ be a complete probability space,
$B(t)$ be a $d$-dimensional G-Brownian motion and
$\mathcal{F}_t=\sigma\{B(s):0\leq s\leq t\}$ be the natural
filtration. Let the filtration $\{\mathcal{F};t\geq0\}$ satisfies
the usual conditions. Let $\mathcal{P}$ be the collection of all
probability measures on $(C_q,\mathcal{B}(C_q))$ and $L_b(C_q)$ be
the set of all bounded continuous functionals. Let $N_0$ be the set
of probability measures on $(-\infty,0]$ such that for any $\mu\in
N_0$, $\int_{-\infty}^0\mu(d\alpha)=1$. For any $m>0$ we define
$N_m$ by
\begin{equation*} N_m=\{\mu\in N_0:\mu^{(m)}=\int_{-\infty}^0 e^{-m\alpha}\mu(d\alpha)<\infty
\},
\end{equation*}
where for any $m\in(0,m_0)$, $N_{m_0}\subset N_{m}\subset N_{0}$
\cite{wym}. Let $g: C_q((-\infty,0];\R^d)\rightarrow\R^d$, $h:
C_q((-\infty,0];\R^{d})\rightarrow\R^{d\times m}$ and
$\gamma:C_q((-\infty,0];\R^d)\rightarrow\R^{d\times m}$ be Borel
measurable. Consider the following stochastic functional
differential equation driven by G-Brownian motion with infinite
delay
\begin{equation}\label{1} dX(t) = g ( X_{t}) dt + h ( X_{t}) d \langle B, B
\rangle (t) + \gamma ( X_{t}) dB (t),
\end{equation}
on $t\geq 0$ with the given initial data $X_0=\zeta\in
C_q((-\infty,0];\R^{d})$ and $ X_{t} = \{ X(t + \alpha) : -\infty <
\alpha \leq 0\}$.
\begin{defn} A continuous $\R^d$-valued and $\mathcal{F}_t$ adapted
process $X(t)$, $-\infty <t< \theta_e$ is called a local strong
solution of problem \eqref{1} with initial data $\zeta\in
C_q((-\infty,0];\R^{d})$ if $X(t)=\zeta(t)$ on $-\infty <t\leq 0$
and for all $t\geq 0$,
\begin{equation*} X(t) =\zeta(0)+ \int_0^{t\wedge\theta_m}g ( X_{s}) ds + \int_0^{t\wedge\theta_m}h ( X_{s}) d \langle B, B
\rangle (s) + \int_0^{t\wedge\theta_m}\gamma ( X_{s}) dB (s),
\end{equation*}
holds q.s. for each $m\geq1$, where $\{\theta_m; m\geq1\}$ is a
nondecreasing sequence of stopping times such that
 $\theta_m\rightarrow\theta_e$ quasi-surely as $m\rightarrow\infty$.
\end{defn}
In addition, if $\lim\sup_{t\rightarrow\theta_e}|X(t)|=\infty$ holds
q.s. when $\theta_e<\infty$ q.s., then $X(t)$, $-\infty <t<
\theta_e$ is called a maximal local strong solution and $\theta_e$
is called the explosion time. If $\theta_e=\infty$, then it is
called a global solution. A maximal local strong solution $X(t)$,
$-\infty <t< \theta_e$ is said to be unique if for any other maximal
local strong solution $Y(t)$, $-\infty <t< \hat{\theta}_e$, we have
$\theta_e=\hat{\theta}_e$ and $X(t)=Y(t)$ for $-\infty <t< \theta_e$
quasi-surely. The rest of the paper is organized as follows. Section
2 is devoted to some basic concepts required for the subsequent
sections of this paper. Section 3 introduces the existence and
uniqueness theory of local and global solutions for stochastic
functional differential equations driven by G-Brownian motion with
infinite delay. Section 4 describes that G-SFDEwID has a bounded
solution. Moreover, it shows that two solutions of G-SFDEwID with
distinct initial data converge. Section 5 studies the asymptotic
properties such as boundedness and convergence of the solutions map
$X_t$ of G-SFDEwID. The $L^2_G$ and exponential estimates are
included in section 6.
\section{Preliminaries}
Building on the concepts of G-Brownian motion theory, this section
includes the basic notions, results and definitions needed for the
further study of the subject. For more details on the concepts
briefly discussed, readers are refer to the papers
\cite{bl,hp,lp,lw,p1,p2,s,rf}. Let $\Omega$ be a given basic
non-empty set. Assume $\mathcal {H}$ be a space of real functions
defined on $\Omega$. Then $(\Omega,\mathcal {H},\hat{\mathbb{E}})$
is a sublinear expectation space, where $\hat{\mathbb{E}}$ is a
sub-expectation defined as the following.
\begin{defn} A functional $\hat{\mathbb{E}}:\mathcal {H}\rightarrow \R$ satisfying
the following four characteristics is known as a sub-expectation.
Let $X,Y\in \mathcal {H}$, then
\begin{itemize}
 \item[${\bf(1)}$] Monotonicity:
$\hat{\mathbb{E}}[X]\geq\hat{\mathbb{E}}[Y]$ if  $X\geq Y$.
\item[${\bf(2)}$] Constant preservation:  $\hat{\mathbb{E}}[K]=K$, for all $K\in\R$.
\item[${\bf(3)}$] Positive homogeneity:
$\hat{\mathbb{E}}[\alpha X]=\alpha \hat{\mathbb{E}}[X]$, for all
$\alpha\in\R^+$.
\item[${\bf(4)}$] Sub-additivity:  $
\hat{\mathbb{E}}[X]+\hat{\mathbb{E}}[Y]\geq \hat{\mathbb{E}}[X+Y]$.
\end{itemize}
\end{defn}
Let $\hat{\mathbb{E}}[Y]=\hat{\mathbb{E}}[-Y]=0$, $K\in\R$ and
$\alpha\in\R^+$ then $\hat{\mathbb{E}}[K+\alpha
Y+X]=K+\hat{\mathbb{E}}[X]$. Furthermore, assume that $\Omega$ be
the space of all $\R^d$-valued continuous paths $(w(t))_{t\geq 0}$
starting from zero equipped with the norm
\begin{equation*}\rho(w^1,w^2)=\sum_{i=1}^{\infty}\frac{1}{2^i}\Big(\max_{t\in[0,i]}|w^1(t)-w^2(t)|\wedge1\Big),\end{equation*}
then for any fixed $T\in[0,\infty)$,
\begin{equation*}L^0_{ip}(\Omega_T)=\Big\{\phi(B(t_1),B(t_2),...,B(t_d)):d\geq1,t_{1},t_{2},...,t_{d}\in[0,T],\phi\in
C_{b.Lip}(\R^{d\times n}))\Big\},\end{equation*} where
$C_{b.Lip}(\R^{d})$ is a space of bounded Lipschitz functions, for
$w\in\Omega$, $t\geq 0$, $B(t)=B(t,w)=w(t)$ is the canonical
process, $L^0_{ip}(\Omega_t)\subseteq L^0_{ip}(\Omega_T)$ for $t\leq
 T$ and
$L^0_{ip}(\Omega)=\cup_{n=1}^{\infty}L^0_{ip}(\Omega_n)$. The
completion of $L^0_{ip}(\Omega)$ under the
 Banach norm $\hat{\mathbb{E}}[|.|^p]^{\frac{1}{p}}$, $p\geq 1$ is denoted by
 $L^p_{G}(\Omega)$, where
$L_{G}^p(\Omega_t)\subseteq L_{G}^p(\Omega_T)\subseteq
L_{G}^p(\Omega)$ for $0\leq t\leq T <\infty.$ Generated by the
canonical process $\{B(t)\}_{t\geq 0}$, the filtration is given by
$\mathcal {F}_t=\sigma\{B(s), 0\leq s \leq t\}$, $\mathcal
{F}=\{\mathcal {F}_t\}_{t\geq 0}$. Let $\pi_T=\{t_0,t_1,...,t_N\}$,
$0\leq t_0\leq t_1\leq...\leq t_N\leq\infty$ be a partition of
$[0,T].$ Choose $p\geq 1,$ let $M^{p,0}_G(0,T)$ denotes a collection
of the following type processes
\begin{equation}\label{p1}\eta_t(w)=\sum_{i=0}^{N-1}\xi_i(w)I_{[t_i,t_{i+1}]}(t),\end{equation}
where $\xi_i\in L_G^p(\Omega_{t_{i}})$, $i=0,1,...,N-1$. Moreover,
the completion of $M_{G}^{p,0}(0,T)$ with the norm given below is
denoted by $M_{G}^{p}(0,T),$ $p\geq 1$
\begin{equation*}\|\eta\|=\Big\{\int_0^T\hat{\mathbb{E}}[|\eta_s|^p]ds\Big\}^{1/p}.\end{equation*}
\begin{defn}
A d-dimensional stochastic process $\{B(t)\}_{t\geq 0}$ satisfying
the following features is called a G-Brownian motion
\begin{itemize}
\item[${\bf(1)}$] $B(0) = 0.$
\item[${\bf(2)}$] The increment $B(t+s) - B(t)$ is
$N(0,[s\underline{\sigma}^2,s\bar{\sigma}^2])$-distributed.
\item[${\bf(3)}$] The increment $B({t+s}) - B(t)$ is independent of
$B({t_1}) ,B({t_2}), ........B({t_d}),$  for every $d \in \Z^+$
 and $0\leq t_1\leq t_2\leq ,...,
\leq t_d \leq t.$
\end{itemize}
\end{defn}
\begin{defn} Let $\eta_t\in
M_G^{2,0}(0,T)$ be given by \eqref{p1}. Then the G-It\^{o}'s
integral $I(\eta)$ is defined by
\begin{align*}
I(\eta)=\int_0^T\eta(s)dB^a(s)=\sum_{i=0}^{N-1}\xi_i\Big(B^a({t_{i+1}})-B^a({t_i})\Big).\end{align*}
A mapping $I:M^{2,0}_{G}(0,T)\mapsto L^2_G(\mathcal{F}_T)$  can be
continuously extended to $I:M^2_G(0,T)\mapsto L^2_G(\mathcal{F}_T)$
and for $\eta\in M^2_G(0,T)$ the G-It\^o integral is still defined
by
\begin{align*}
\int_0^T\eta(s)dB^a(s)=I(\eta).\end{align*}
\end{defn}
\begin{defn} The G-quadratic variation process $\{\langle
B^a\rangle(t)\}_{t\geq0}$ of G-Brownian motion is defined by
\begin{align*}\begin{split}&
\langle
B^a\rangle(t)=\lim_{N\rightarrow\infty}\sum_{i=0}^{N-1}\Big(B^a({t_{i+1}^N})-B^a({t_{i}^N})\Big)^2={B^a(t)}^2-2\int_0^tB^a(s)dB^a(s),\\&
\end{split}\end{align*}
which is an increasing process with $\langle B^a\rangle(0)=0$ and
for any $0\leq s\leq t$,
\begin{align*} \langle B^a\rangle(t)-\langle B^a\rangle(s)\leq
\sigma_{aa^\tau} (t-s).
\end{align*}
\end{defn}
Assume that $a, \hat{a} \in\R^d$ be two given vectors. Then the
mutual variation process of $B^a$ and $B^{\hat{a}}$ is defined by
$\langle B^a,B^{\hat{a}}\rangle=\frac{1}{4}[\langle
B^a+B^{\hat{a}}\rangle(t)-\langle B^a-B^{\hat{a}}\rangle(t)]$. A
mapping $H_{0,T}:M^{0,1}_{G}(0,T)\mapsto L^2_G(\mathcal{F}_T)$  is
defined by
\begin{align*} H_{0,T}(\eta)=\int_0^T\eta(s)d\langle B^a\rangle(s)= \sum_{i=0}^{N-1}\xi_i\Big(\langle
B^a\rangle_({t_{i+1}})-\langle B^a\rangle({t_{i}})\Big),
\end{align*}
which can be continuously extended to $M^1_G(0,T)$ and for $\eta\in
M^1_G(0,T)$ this is still denoted by
\begin{align*}
\int_0^T\eta(s)d\langle B^a\rangle(s)=H_{0,T}(\eta).
\end{align*}
The G-It\^{o} integral and its quadratic variation process satisfies
the following properties \cite{p3,wzl}.
\begin{prop}
\begin{itemize}
\item[${\bf(1)}$] $\hat{\mathbb{E}}[\int_0^T\eta(s)dB(s)]=0,\,\,\, \textit{for all}\,\,\, \eta\in M^p_G(0,T)$.
\item[${\bf(2)}$]
$\hat{\mathbb{E}}[(\int_0^T\eta(s)dB(s))^2]=\hat{\mathbb{E}}[\int_0^T\eta^2(s)\langle
B, B\rangle(t)]\leq \bar{\sigma}^2
\mathbb{E}[\int_0^T\eta^2(s)dt],\,\,\, \textit{for all}\,\,\,
\eta\in M^2_G(0,T)$.
\item[${\bf(3)}$] $\hat{\mathbb{E}}[\int_0^T|\eta(s)|^pdt]\leq \int_0^T\hat{\mathbb{E}}|\eta(s)|^pdt,\,\,\,
\textit{for all}\,\,\, \eta\in M^p_G(0,T)$.
\end{itemize}
\end{prop}
The concept of G-capacity and lemma \ref{l2} can be found in
\cite{dhp}.
\begin{defn}
Let $\mathcal {B}(\Omega)$ be a Borel $\sigma$-algebra of $\Omega$
and $\mathcal{P}$ be a collection of all probability measures on
$(\Omega, \mathcal {B}(\Omega)$. Then the G-capacity denoted by
$\hat{C}$ is defined as the following
\begin{equation*}\hat{C}(A)=\sup_{\mathbb{P}\in\mathcal{P}}\mathbb{P}(A),\end{equation*}
where set $A\in\mathcal {B}(\Omega)$.
\end{defn}
\begin{defn}
A set $A\in\mathcal {B}(\Omega)$ is said to be polar if its capacity
is zero i.e. $\hat{C}(A)=0$ and a property holds quasi-surely (q.s)
if it holds outside a polar set.
\end{defn}
\begin{lem}\label{l2} Let $X\in L^p$ and $\hat{\mathbb{E}}|X|^p<\infty$. Then for each
$\delta>0,$ the G-Markov inequality is defined by
\begin{equation*}\hat{C}(|X|>\delta)\leq \frac{\hat{\mathbb{E}}[|X|^p]}{\delta}.\end{equation*}
\end{lem}
For the proof of the following lemmas \ref{l3} and \ref{l4} see
\cite{g}.
\begin{lem}\label{l3} Let $p\geq 2$, $\eta\in M_G^2(0,T)$, $a\in\R^d$ and $X(t)=\int_0^t
\eta(s)dB^a(s)$. Then there exists a continuous modification
$\bar{X}(t)$ of $X(t)$, that is, on some $\bar{\Omega}\subset\Omega$
with $\hat{C}(\bar{\Omega}^c)=0$ and for all $t\in[0,T]$,
$\hat{C}(|X(t)-\bar{X}|\neq 0)=0$ such that
\begin{equation*}
\hat{\mathbb{E}}\Big[\sup_{s\leq v\leq
t}|\bar{X}(v)-\bar{X}(s)|^p\Big]\leq
\hat{K}\sigma_{aa^{\tau}}^{\frac{p}{2}}\hat{\mathbb{E}}\Big(\int_s^t|\eta(v)|^2dv\Big)^{\frac{p}{2}},
\end{equation*}
where $0<\hat{K}<\infty$ is a positive constant.
\end{lem}
\begin{lem}\label{l4} Let $p\geq 1$, $\eta\in M_G^p(0,T)$ and $a,\hat{a}\in\R^d$, then there exists a continuous modification
$\bar{X}^{a,\hat{a}}(t)$ of  $X^{a,\hat{a}}(t)=\int_0^t \eta(s)d
\langle B^{a}, B^{\hat{a}} \rangle (s)$ such that for $0\leq s\leq
t\leq T$,
\begin{equation*}
\hat{\mathbb{E}}\Big[\sup_{0\leq s\leq v\leq
t}|\bar{X}^{a,\hat{a}}(v)-\bar{X}^{a,\hat{a}}(s)|^p\Big]\leq
\Big(\frac{1}{4}\sigma_{(a+\hat{a})(a-\hat{a})^\tau}\Big)^{p}(t-s)^{p-1}\hat{\mathbb{E}}\int_s^t|\eta(v)|^pdv,
\end{equation*}
\end{lem}
The next two lemmas will also be used in the subsequent sections of
this article \cite{m}.
\begin{lem}\label{l7} Let $a,b\geq0$ and
$\epsilon\in(0,1).$ Then
\begin{align*}
(a+b)^2\leq\frac{a^2}{\epsilon}+\frac{b^2}{1-\epsilon}.
\end{align*}
\end{lem}
\begin{lem}\label{l8} Assume $p\geq2$ and $\hat{\epsilon},a,b>0.$
Then the following two inequalities hold.
\begin{itemize}
 \item[${\bf(i)}$]  $a^{p-1}b\leq\frac{(p-1)\hat{\epsilon}a^p}{p}+\frac{b^p}{p\hat{\epsilon}^{p-1}}.$
\item[${\bf(ii)}$]  $a^{p-2}b^2\leq\frac{(p-2)\hat{\epsilon}a^p}{p}+\frac{2b^p}{p\hat{\epsilon}^{\frac{p-2}{2}}}.$
\end{itemize}
\end{lem}
\section{Existence and uniqueness of solutions}
In the phase space $BC((-\infty,T];\R^d))$, equation \eqref{1} under
the global Lipschitz and growth conditions admits a unique solution
\cite{f6,rbs}. In this section, we study the existence and
uniqueness theory for the solutions to \eqref{1} in the phase space
with fading memory $C_q((-\infty,0];\R^d)$.
\begin{thm}\label{thn} Let there exists two positive constants $L$ and $K$ such that for all
$\varphi,\psi\in C_q((-\infty,0];\R^d)$ and $t\in[0,T]$, the following conditions
hold.
\begin{equation}\label{cf1}|g(\psi)-g(\varphi)|^2\vee |h(\psi)-h(\varphi)|^2\vee|\gamma(\psi)-\gamma(\phi)|^2\leq L\|\psi-\phi\|_q^2,\end{equation}
\begin{equation}\label{cf2}|g(\varphi)|^2\vee |h(\varphi)|^2\vee|\gamma(\varphi)|^2\leq K(1+\|\varphi\|_q^2).\end{equation}
Then problem \eqref{1} with initial data $\zeta\in
C_q((-\infty,0];\R^{d})$ admits a unique bounded solution $X(t)$,
which is continuous and $\mathcal{F}_t$-adapted on $t\in[0,T]$.
\end{thm}
We omit the proof as it can be derived in a similar way like
\cite{rbs}. To extend the above existence-uniqueness result to
G-SFDEwID with locally Lipschitz continuous coefficients, we first
give the following lemma.
\begin{lem}\label{Lf3} Let  $p\geq 1$ and $\lambda< pq$. Then for any $\zeta\in C_q((-\infty,0];\R^d)$,
\begin{equation*}\begin{split}\label{}
\|X_t\|^p_q&\leq  e^{-\lambda t} \|\zeta\|_q^p+\sup_{0<s\leq
t}|X(s)|^p.
\end{split}\end{equation*}
\end{lem}
\begin{proof} By virtue of the definition of norm $\|.\|$ and observing that
$pq>\lambda$ we have
\begin{equation*}\begin{split}\label{}
\|X_t\|^p_q& = \Big[\sup_{-\infty<\alpha\leq
0}e^{q\alpha}|X(t+\alpha)|\Big]^p\\& \leq \sup_{-\infty<\alpha\leq
0}e^{\lambda\alpha}|X(t+\alpha)|^p\\& \leq \sup_{-\infty<s\leq
0}e^{-\lambda(t-s)}|X(s)|^p+\sup_{0<s\leq
t}e^{-\lambda(t-s)}|X(s)|^p\\& =e^{-\lambda t}
\|\zeta\|_q^p+e^{-\lambda t}\sup_{0<s\leq t}e^{\lambda
s}|X(s)|^p\\&\leq e^{-\lambda t} \|\zeta\|_q^p+\sup_{0<s\leq
t}|X(s)|^p.
\end{split}\end{equation*}
The proof is complete.
\end{proof}
All through this article we assume that for any $p\geq 1$, $\lambda<
pq$.
\begin{thm}\label{thm1} Let for any $m>0$ there exists a positive constant $K_m$ such that for all $\varphi,\psi\in C_q((-\infty,0];\R^d)$
and $t\in[0,T]$, the following local Lipschitz condition hold,
\begin{equation}\label{cn}
|g(\varphi)-g(\psi)|^2\vee|h(\psi)-h(\varphi)|^2\vee|\gamma(\psi)-\gamma(\varphi)|^2\leq
K_m\|\psi-\phi\|^2_q,
\end{equation}
with $\|\psi\|\vee\|\varphi\|\leq m.$ Then the G-SFDEwID \eqref{1}
having the initial data $\zeta\in C_q((-\infty,0];\R^{d})$ admits a
continuous and $\mathcal{F}_t$-adapted unique local solution $X(t)$
quasi-surely on $t\in(-\infty,\theta_e)$, where $\theta_e$ is the
potential explosion time.
\end{thm}
\begin{proof}
For any $m\geq m_0$, we define the following stopping time
\begin{equation*} \theta_m=\inf\{t\geq0, |X(t)|>m\},
\end{equation*}
with $\inf\emptyset=\infty.$ Let equation \eqref{1} has two
solutions $X(t)$ and $Y(t)$ where
\begin{equation*}
\theta_m=\inf\{t\geq0, |X(t)|>m\} \wedge \inf\{t\geq0, |Y(t)|>m\}.
\end{equation*}
By using the inequality $(\sum_{i=1}^3a_i)^2\leq3\sum_{i=1}^3a^2_i$,
from \eqref{1} we have
\begin{equation*}\begin{split}
|Y(t\wedge \theta_m)-X(t\wedge
\theta_m)|^2&\leq3\Big|\int_0^{s\wedge
\theta_m}[g(Y_s)-g(X_s)]ds\Big|^2+3\Big|\int_0^{s\wedge
\theta_m}[h(Y_s)-h(X_s)]d \langle B, B \rangle
(s)\Big|^2\\&+3\Big|\int_0^{s\wedge
\theta_m}[\gamma(Y_s)-\gamma(X_s)]d B(s)\Big|^2.
\end{split}\end{equation*}
Taking the G-expectation on both sides, using the Holder inequality,
lemma \ref{l3}, lemma \ref{l4} and condition \eqref{cn}, there exist
positive constants $c$, $c_1$ and $c_2$ such that
\begin{equation*}\begin{split}
\hat{\mathbb{E}}\Big[\sup_{0\leq s\leq t}|Y(s\wedge
\theta_m)-X(s\wedge
\theta_m)|^2\Big]&\leq3c\hat{\mathbb{E}}\int_0^{t\wedge
\theta_m}|g(Y_s)-g(X_s)|^2ds+3c_1\hat{\mathbb{E}}\int_0^{t\wedge
\theta_m}|h(Y_s)-h(X_s)|^2d s\\&+3c_2\hat{\mathbb{E}}\int_0^{t\wedge
\theta_m}[\gamma(Y_s)-\gamma(X_s)|^2d s\\& \leq
3K_m(c+c_1+c_2)\hat{\mathbb{E}}\int_0^{t\wedge
\theta_m}\|Y_s-X_s\|_q^2d s\\& \leq
3K_m(c+c_1+c_2)\int_0^{t}\hat{\mathbb{E}}\Big[\sup_{0\leq s\leq
t}|Y(s\wedge \theta_m)-X(s\wedge \theta_m)|^2\Big]d s
\end{split}\end{equation*}
The above last inequality is obtained by using lemma \eqref{Lf3}.
Finally, the Grownwall inequity yields
\begin{equation*}
\hat{\mathbb{E}}\Big[\sup_{0\leq s\leq t}|Y(s\wedge
\theta_m)-X(s\wedge \theta_m)|^2\Big]=0,
\end{equation*}
which gives that $Y(t\wedge \theta_m)=X(t\wedge \theta_m)$
quasi-surely. According to the definition of $\theta_m$,
$\{\theta_m:m\geq m_0\}$ is a non-decreasing sequence and as
$m\rightarrow \infty$ quasi-surely $\theta_m\rightarrow
\theta_{\infty}\leq \theta_e$. We therefore have $Y(t)=X(t)$ for all
$t\in(-\infty,\theta_e)$. The uniqueness has been proved. To prove
the existence of solutions, for any sufficiently large $m_0$ and
$m\geq m_0,$ we set the truncation functions $g_m$, $h_m$ and
$\gamma_m$ as follows
\begin{equation*}\begin{split}&g_m(\varphi)= \begin{cases}
    g(\varphi), & \text{if $\|\varphi\|_q\leq m$ ;}
    \\    g(\frac{m\varphi}{\|\varphi\|_q}), & \text{if $\|\varphi\|_q> m$,}
  \end{cases}\\&
h_m(\varphi)= \begin{cases}
    h(\varphi), & \text{if $\|\varphi\|_q\leq m$ ;}
    \\    h(\frac{m\varphi}{\|\varphi\|_q}), & \text{if $\|\varphi\|_q> m$,}
  \end{cases}\\&
  \gamma_m(\varphi)= \begin{cases}
    \gamma(\varphi), & \text{if $\|\varphi\|_q\leq m$ ;}
    \\    \gamma(\frac{m\varphi}{\|\varphi\|_q}), & \text{if $\|\varphi\|_q> m$.}
  \end{cases}\\&
  \end{split}\end{equation*}
Then $g_m$, $h_m$ and $\gamma_m$ satisfy the assumptions \eqref{cf1}
and \eqref{cf2}. By virtue of theorem \ref{thn}, problem
\begin{equation}\label{n} X^{(m)}(t) = \zeta(0)+\int_0^t g_m ( X^{(m)}_{s}) ds +\int_0^t  h_m ( X^{(m)}_{s}) d \langle B, B
\rangle (s) + \int_0^t \gamma_m ( X^{(m)}_{s}) dB (s),
\end{equation}
admits a unique bounded solution $X^{(m)}(t)$, which is continuous
and $\mathcal{F}_t$-adopted. Notice that for $0\leq t\leq\theta_m$,
we have $g_n(\varphi^n)=g(\varphi^n)$, $h_n(\varphi^n)=h(\varphi^n)$
and $\gamma_n(\varphi^n)=\gamma(\varphi^n)$. By using the inequality
$(\sum_{i=1}^3a_i)^2\leq3\sum_{i=1}^3a^2_i$, from \eqref{1} and
\eqref{n} for each $t\in[0,\theta_m]$, we get
\begin{equation*}\begin{split}
|X^{(m)}(t)-X(t)|^2&\leq3\Big|\int_0^{t}[g(X^{(m)}_s)-g(X_s)]ds\Big|^2+3\Big|\int_0^{t}[h(X^{(m)}_s)-h(X_s)]d
\langle B, B \rangle
(s)\Big|^2\\&+3\Big|\int_0^{t}[\gamma(X^{(m)}_s)-\gamma(X_s)]d
B(s)\Big|^2.
\end{split}\end{equation*}
Taking the G-expectation on both sides, using the Holder inequality,
lemma \ref{l3}, lemma \ref{l4}, condition \eqref{cn} and lemma
\eqref{Lf3}, by straightforward calculations in a similar way as
above, we derive
\begin{equation*}\begin{split}
\hat{\mathbb{E}}\Big[\sup_{0\leq s\leq
t}|X^{(m)}(s)-X(s)|^2\Big]&\leq3c\hat{\mathbb{E}}\int_0^{t}|g(X^{(m)}_s)-g(X_s)|^2ds+3c_1\hat{\mathbb{E}}\int_0^{t}|h(X^{(m)}_s)-h(X_s)|^2d
s\\&+3c_2\hat{\mathbb{E}}\int_0^{t}[\gamma(X^{(m)}_s)-\gamma(X_s)|^2d
s\\& \leq
3K_n(c+c_1+c_2)\hat{\mathbb{E}}\int_0^{t}\|X^{(n)}_s-X_s\|_q^2d s\\&
\leq 3K_m(c+c_1+c_2)\int_0^{t}\hat{\mathbb{E}}\Big[\sup_{0\leq s\leq
t}|X^{(m)}(s)-X(s)|^2\Big]d s\\&
\end{split}\end{equation*}
By virtue of the Grownwall inequity, it follows
\begin{equation}\label{n1}
\hat{\mathbb{E}}\Big[\sup_{0\leq s\leq
t}|X^{(m)}(s)-X(s)|^2\Big]=0,\,\,\,\,\,0\leq t\leq \theta_m,
\end{equation}
The above expression means that for all $t\in[0,\theta_m]$,
$X^{(m)}(t)=X(t)$ quasi-surely. Therefore we have $X^{(m)}(t)=X(t)$,
for all $t\in(-\infty,\theta_m]$ quasi-surely. Also, by using lemma
\ref{Lf3} and expression \eqref{n1} we have
\begin{equation*}
\hat{\mathbb{E}}\|X^{(m)}_t-X_t\|^2_q\leq\hat{\mathbb{E}}\Big[\sup_{0\leq
s\leq t}|X^{(m)}(s)-X(s)|^2\Big]=0,\,\,\,\,\,0\leq t\leq \theta_m,
\end{equation*}
which implies that for all $t\in[0,\theta_m]$, $X^{(m)}_t=X_t$
quasi-surely. Next to show that $X(t)$ is the solution of \eqref{1},
by $X^{(m)}(t\wedge \theta_m)=X(t\wedge \theta_m)$,
$X^{(m)}_{t\wedge \theta_m}=X_{t\wedge \theta_m}$ and \eqref{n}, it
follows
\begin{equation*}\begin{split} X(t\wedge \theta_m) &= \zeta(0)+\int_0^{t\wedge \theta_m} g_m ( X_{s}) ds
+\int_0^{t\wedge \theta_m}  h_m ( X_{s}) d \langle B, B \rangle (s)
+ \int_0^{t\wedge \theta_m} \gamma_m ( X_{s}) dB (s)\\& =
\zeta(0)+\int_0^{t\wedge \theta_m} g ( X_{s}) ds +\int_0^{t\wedge
\theta_m}  h ( X_{s}) d \langle B, B \rangle (s) + \int_0^{t\wedge
\theta_m} \gamma ( X_{s}) dB (s),
\end{split}\end{equation*}
 letting $m\rightarrow\infty$, it follows that for
any $t\in[0,\theta_e)$,
\begin{equation*}\begin{split} X(t) & =
\zeta(0)+\int_0^{t} g ( X_{s}) ds +\int_0^{t}  h ( X_{s}) d \langle
B, B \rangle (s) + \int_0^{t} \gamma ( X_{s}) dB (s),
\end{split}\end{equation*}
that is, $X(t)$, $t\in(-\infty,\theta_e)$ is a local solution of
equation \eqref{1}. The proof stands completed.
\end{proof}
To examine the global existence and uniqueness of solutions for
problem \eqref{1}, we assume the following conditions.
\begin{itemize}
\item[${\bf(A_{1})}$] For any probability measure $\mu_1,\mu_2,\mu_3\in N_{2q}$
there exists positive constants $\lambda_i$, $i=1,2,..,5$ such that
for any $\psi,\varphi\in C_q((-\infty,0];\R^d)$, we have
\begin{equation}\label{c1}
[\psi(0)-\varphi(0)]^{\tau}[g(\psi)-g(\varphi)]\leq
-\lambda_1|\psi(0)-\varphi(0)|^2+\lambda_2\int_{-\infty}^0|\psi(\alpha)-\varphi(\alpha)|^2\mu_1(d\alpha),
\end{equation}
\begin{equation}\label{c2}
[\psi(0)-\varphi(0)]^{\tau}[h(\psi)-h(\varphi)]\leq
-\lambda_3|\psi(0)-\varphi(0)|^2+\lambda_4\int_{-\infty}^0|\psi(\alpha)-\varphi(\alpha)|^2\mu_2(d\alpha),
\end{equation}
and
\begin{equation}\label{c3}
|\gamma(\psi)-\gamma(\varphi)|^2\leq
\lambda_5\int_{-\infty}^0|\psi(\alpha)-\varphi(\alpha)|^2\mu_3(d\alpha).
\end{equation}
\end{itemize}
The upcoming lemma will be used in several places all through this
article.
\begin{lem}\label{Lf2} Let $p\geq 2$, $\lambda< pq$ and for any
$i\in \Z^+$, $\mu_i\in N_{l}$. Then for any $\zeta\in
C_q((-\infty,0];\R^d)$,
\begin{equation}\label{fn}
\int_0^{t}\int_{-\infty}^0|X(s+\alpha)|^p\mu_i(d\alpha)ds \leq
\frac{\mu_i^{(2q)}}{2q}\|\zeta\|^p_q
+\int_{0}^{t}|X(s)|^pds,\end{equation}
\begin{equation}\label{ffn}
\int_0^{t}\int_{-\infty}^0e^{\lambda
s}|X(s+\alpha)|^p\mu_i(d\alpha)ds \leq
\frac{\mu_i^{(pq)}}{2q-\lambda}\|\zeta\|^p_q
+\mu_i^{(pq)}\int_{0}^{t}e^{\lambda s}|X(s)|^pds.\end{equation}
\end{lem}
\begin{proof} Noticing that $\zeta\in C_q((-\infty,0];\R^d)$ and for any
$i\in \Z^+$, $\mu_i\in N_{pq}$, using the definition of norm and the
Fubini theorem, we derive
\begin{equation*}\begin{split}
&\int_0^{t}\int_{-\infty}^0|X(s+\alpha)|^p\mu_i(d\alpha)ds\\&
=\int_0^{t}\Big[\int_{-\infty}^{-s}e^{pq(s+\alpha)}|X(s+\alpha)|^pe^{-pq(s+\alpha)}\mu_i(d\alpha)+\int_{-s}^0|X(s+\alpha)|^p\mu_i(d\alpha)\Big]ds\\&
\leq \|\zeta\|^p_q \int_0^{t}
e^{-pqs}ds\int_{-\infty}^{0}e^{-pq\alpha}\mu_i(d\alpha)+\int_{-\infty}^0\mu_i(d\alpha)\int_0^{t}|X(s)|^pds,
\end{split}\end{equation*}
by noticing that $\int_{-\infty}^0\mu_i(d\alpha)=1$ and
$\int_{-\infty}^{0}e^{-pq\alpha}\mu_i(d\alpha)=\mu_i^{(pq)}$,
$i\in\Z^+$, we derive
\begin{equation*}
\int_0^{t}\int_{-\infty}^0|X(s+\alpha)|^p\mu_i(d\alpha)ds \leq
\frac{\mu_i^{(pq)}}{pq}\|\zeta\|^p_q
+\int_{0}^{t}|X(s)|^pds.\end{equation*} The proof of \eqref{fn} is
complete. To prove \eqref{ffn}, we use similar arguments as used
above and proceed as follows
\begin{equation*}\begin{split}
&\int_0^{t}\int_{-\infty}^0e^{\lambda
s}|X(s+\alpha)|^p\mu_i(d\alpha)ds\\& =\int_0^{t}e^{\lambda
s}ds\Big[\int_{-\infty}^{-s}|z(s+\alpha)|^p\mu_i(d\alpha)+\int_{-s}^{0}|X(s+\alpha)|^p\mu_i(d\alpha)\Big]\\&
=\int_0^{t}e^{\lambda
s}ds\int_{-\infty}^{-s}|X(s+\alpha)|^p\mu_i(d\alpha)+\int_{-t}^{0}\mu_i(d\alpha)\int_{-\alpha}^{t}e^{\lambda
s}|X(s+\alpha)|^pds\\& \leq \int_0^{t}e^{\lambda
s}ds\int_{-\infty}^{-s}e^{2q(s+\alpha)}|X(s+\alpha)|^pe^{-2q(s+\alpha)}\mu_i(d\alpha)+\int_{-\infty}^{0}\mu_i(d\alpha)\int_{0}^{t}e^{\lambda
(s-\alpha)}|X(s)|^pds\\& \leq \|\zeta\|^p_q
\int_0^{t}e^{-(pq-\lambda)
s}ds\int_{-\infty}^{0}e^{-pq\alpha}\mu_i(d\alpha)+\int_{-\infty}^{0}e^{-\lambda
\alpha}\mu_i(d\alpha)\int_{0}^{t}e^{\lambda s}|X(s)|^pds,
\end{split}\end{equation*}
by using the definition $\mu^{(m)}=\int_{-\infty}^{0}e^{-m
\alpha}\mu(d\alpha)$ and noticing that $pq>\lambda$, we have
\begin{equation*}
\int_0^{t}\int_{-\infty}^0e^{\lambda
s}|X(s+\alpha)|^p\mu_i(d\alpha)ds \leq
\frac{\mu_i^{(2q)}}{pq-\lambda}\|\zeta\|^2_q
+\mu_i^{(pq)}\int_{0}^{t}e^{\lambda s}|X(s)|^pds.\end{equation*} The
proof of \eqref{ffn} stands completed.
\end{proof}
\begin{thm} Let assumptions \eqref{cn} and $A_{1}$ hold. Then the G-SFDEwID \eqref{1} has a continuous and $\mathcal{F}_t$-adapted global
solution.
\end{thm}
\begin{proof} For any initial data $\zeta\in C_q$, in view of local
Lipschitz condition $A_1$, theorem \ref{thm1} gives that \eqref{1}
admits a unique maximal local strong solution $X(t)$ on
$t\in(-\infty,\theta_e)$ and this solution is continuous for any
$t\in(-\infty,\theta_e)$ and $\mathcal{F}_t$-adopted. To prove that
this solution is global, we only need to show that $\theta_e=\infty$
q.s. Note that $\theta_m$ is increasing as $m\rightarrow\infty$ and
$\theta_m\rightarrow \theta_\infty\leq\theta_e$ q.s. If we can prove
that $\theta_\infty=\infty$ q.s., then $\theta_e=\infty$ q.s., which
implies that $X(t)$ is global. This is equivalent to proving that as
$m\rightarrow\infty$, for any $T>0$, $\hat{C}(\theta_m\leq
T)\rightarrow0$. Applying the G-It\^o formula to $|X(t)|^2$, taking
G-expectation on both sides, using properties of G-It\^o integral
and lemma \ref{l4}, there exists a positive constant $k_1$ such that
\begin{equation}\begin{split}\label{3} \hat{\mathbb{E}}|X(t\wedge\theta_m)|^2&\leq
\hat{\mathbb{E}}|X(0)|^2+\hat{\mathbb{E}}\int_0^{t\wedge\theta_m}2X^\tau(s)
g(X_s)ds\\&+k_1\hat{\mathbb{E}}\int_0^{t\wedge\theta_m}\Big[2X^\tau(s)
h(X_s)+|\gamma(X_s)|^2\Big]ds.
\end{split}\end{equation}
By using condition \eqref{c1} and the fundamental inequality
$2a_1a_2\leq \sum_{i=1}^2 a^2_i$ we derive
\begin{equation}\begin{split}\label{4}
X^{\tau}(t)g(X_t)&\leq
-(\lambda_1-\frac{1}{2})|X(t)|^2+\frac{1}{2}|g(0)|^2+\lambda_2\int_{-\infty}^0|X(t+\alpha)|^2\mu_1(d\alpha).
\end{split}\end{equation}
Similar arguments follows
\begin{equation}\begin{split}\label{5}
X^{\tau}(t)h(X_t)&\leq
-(\lambda_3-\frac{1}{2})|X(t)|^2+\frac{1}{2}|h(0)|^2+\lambda_4\int_{-\infty}^0|X(t+\alpha)|^2\mu_2(d\alpha).
\end{split}\end{equation}
By using condition\eqref{c3} and the basic inequality $(\sum_{i=1}^2
a_i)^2\leq 2\sum_{i=1}^2 a^2_i$ we get
\begin{equation}\begin{split}\label{6}
|\gamma(X_t)|^2&\leq2|\gamma(0)|^2+2\lambda_5\int_{-\infty}^0|X(t+\alpha)|^2\mu_3(d\alpha).
\end{split}\end{equation}
On substituting \eqref{4}, \eqref{5} and \eqref{6} in \eqref{3}, we
derive
\begin{equation*}\begin{split}
\hat{\mathbb{E}}|X(t\wedge\theta_m)|^2& \leq
\hat{\mathbb{E}}|X(0)|^2+[|g(0)|^2+k_1|h(0)|^2+2k_1|\gamma(0)|^2]T\\&+(k_1-2\lambda_1-2k_1\lambda_3+1)\hat{\mathbb{E}}\int_0^{t\wedge\theta_m}|X(s)|^2ds
+2\lambda_2\hat{\mathbb{E}}\int_0^{t\wedge\theta_m}\int_{-\infty}^0|X(s+\alpha)|^2\mu_1(d\alpha)ds\\&
+2k_1\lambda_4\hat{\mathbb{E}}\int_0^{t\wedge\theta_m}\int_{-\infty}^0|X(s+\alpha)|^2\mu_2(d\alpha)ds\\&
+2k_1\lambda_5\hat{\mathbb{E}}\int_0^{t\wedge\theta_m}\int_{-\infty}^0|X(s+\alpha)|^2\mu_3(d\alpha)ds,
\end{split}\end{equation*}
letting
$K_1=\hat{\mathbb{E}}|X(0)|^2+[|g(0)|^2+k_1|h(0)|^2+2k_1|\gamma(0)|^2]T$,
we obtain
\begin{equation}\begin{split}\label{7}
\hat{\mathbb{E}}|X(t\wedge\theta_m)|^2&\leq
K_1+(-2\lambda_1-2k_1\lambda_3+k_1+1)\hat{\mathbb{E}}\int_0^{t\wedge\theta_m}|X(s)|^2ds\\&
+2\lambda_2\hat{\mathbb{E}}\int_0^{t\wedge\theta_m}\int_{-\infty}^0|X(s+\alpha)|^2\mu_1(d\alpha)ds
+2\lambda_4k_1\hat{\mathbb{E}}\int_0^{t\wedge\theta_m}\int_{-\infty}^0|X(s+\alpha)|^2\mu_2(d\alpha)ds\\&
+2k_1\lambda_5\hat{\mathbb{E}}\int_0^{t\wedge\theta_m}\int_{-\infty}^0|X(s+\alpha)|^2\mu_3(d\alpha)ds.
\end{split}\end{equation}
From result \eqref{fn} of lemma \ref{Lf2}, we have
\begin{equation}\label{f}
\int_0^{t\wedge\theta_m}\int_{-\infty}^0|X(s+\alpha)|^2\mu_i(d\alpha)ds
\leq \frac{1}{2q}\|\zeta\|^2_q\mu_i^{(2q)}
+\int_{0}^{t}|X(s\wedge\theta_m)|^2ds.\end{equation} Substitutions
for $i=1,2,3$ in \eqref{7} give
\begin{equation*}\begin{split}
\hat{\mathbb{E}}|X(t\wedge\theta_m)|^2&\leq
K_1+\frac{1}{q}[\lambda_2\mu_1^{(2q)}+k_1\lambda_4\mu_2^{(2q)}+k_1\lambda_5\mu_3^{(2q)}]\hat{\mathbb{E}}\|\zeta\|^2_q\\&
+(k_1+1-2\lambda_1+2\lambda_2-2k_1\lambda_3+2k_1\lambda_4+2k_1\lambda_5)\hat{\mathbb{E}}\int_0^{t}|X(s\wedge\theta_m)|^2ds\\&
= K_2+ K_3\int_0^{t}\hat{\mathbb{E}}|X(s\wedge\theta_m)|^2ds
\end{split}\end{equation*}
where
$K_2=K_1+\frac{1}{q}[\lambda_2\mu_1^{(2q)}+k_1\lambda_4\mu_2^{(2q)}+k_1\lambda_5\mu_3^{(2q)}]\|\zeta\|^2_q$
and
$K_3=k_1+1-2\lambda_1+2\lambda_2-2k_1\lambda_3+2k_1\lambda_4+2k_1\lambda_5$.
By virtue of the Grownwall inequality,
\begin{equation*}
\hat{\mathbb{E}}|X(t\wedge\theta_m))|^2\leq K_2e^{K_3t},
\end{equation*}
taking $t=T$ yields
\begin{equation*}
\hat{\mathbb{E}}|X(T\wedge\theta_m)|^2\leq K_2e^{K_3T}.
\end{equation*}
By using lemma \ref{l2}, the definition of $\theta_m$ and the above
inequality, it follows
\begin{equation*}\begin{split}
\hat{C}(\theta_m\leq T)&=\hat{C}(\theta_m\leq
T,|X(T\wedge\theta_m)|>m)\\&\leq\frac{1}{m^2}\hat{\mathbb{E}}|X(\theta_m)|^21_{\theta_m\leq
T}\\&=\frac{1}{m^2}\hat{\mathbb{E}}|X(T\wedge\theta_m)|^21_{\theta_m\leq
T}\\& =\frac{1}{m^2}\hat{\mathbb{E}}|X(T\wedge\theta_m)|^2\\& \leq
\frac{1}{m^2}.K_2e^{K_3T}.
\end{split}\end{equation*}
Taking limits $m\rightarrow\infty$ gives
\begin{equation*}
\lim_{m\rightarrow\infty}\hat{C}(\theta_m\leq T)=0,
\end{equation*}
which implies that G-SFDEwID \eqref{1} admits a unique global
solution $X(t)$ on $(-\infty,\infty)$ quasi-surely. The proof stands
completed.
\end{proof}
\section{Asymptotic properties of solution}
In this section, the mean square boundedness for solution of the
G-SFDEwID is proved. Under different initial data, the convergence
of solutions is derived.
\begin{thm}\label{thm2} Let $X(t)$ be the unique solution of problem \eqref{1} with initial data $\zeta\in C_q((-\infty,0];\R^d)$.
 Assume assumption $A_{1}$ holds. Let
$\lambda_i$, $i=1,2,..,5$ satisfy $2\lambda_1>
2\lambda_2\mu_1^{(2q)}+2k_1\lambda_4\mu_2^{(2q)}+k_1\lambda_5\mu_3^{(2q)}-2k_1\lambda_3$.
Then there exists $\lambda\in (0,(2\lambda_1+2k_1\lambda_3
-2\lambda_2\mu_1^{(2q)}-2k_1\lambda_4\mu_2^{(2q)}-k_1\lambda_5\mu_3^{(2q)})\wedge
2q)$ such that
\begin{equation}\label{f3}
\hat{\mathbb{E}}[|X(t)|^2]\leq K_4+K_5e^{-\lambda t},\end{equation}
where
\begin{equation*}K_4=\frac{1}{\lambda}\Big(\frac{1}{\epsilon}|g(0)|^2+\frac{k_1}{\epsilon_1}|h(0)|^2+\frac{k_1}{\epsilon_2}|\gamma(0)|^2\Big)\end{equation*}
and \begin{equation*}K_5=\hat{\mathbb{E}}|X(0)|^2+
\frac{2\lambda_2\mu_1^{(2q)}}{2q-\lambda}\hat{\mathbb{E}}\|\zeta\|_q^2
+
\frac{2k_1\lambda_4\mu_2^{(2q)}}{2q-\lambda}\hat{\mathbb{E}}\|\zeta\|_q^2
+\frac{k_1\lambda_5\mu_3^{(2q)}}{(2q-\lambda)(1-\epsilon_2)}\hat{\mathbb{E}}\|\zeta\|_q^2.\end{equation*}
and $\epsilon$,$\epsilon_1$ and $\epsilon_2$ are sufficiently small
such that
\begin{equation*}2\lambda_1-\epsilon-\lambda-k_1\epsilon_1+2k_1\lambda_3
-2\lambda_2\mu_1^{(2q)}-2k_1\lambda_4\mu_2^{(2q)}-\frac{k_1\lambda_5}{1-\epsilon_2}\mu_3^{(2q)}>0.
\end{equation*}
\end{thm}
\begin{proof}
Applying the G-It\^o formula to $e^{t\lambda}|X(t)|^2$, taking the
G-expectation on both sides, using properties of G-It\^o integral
and lemma \ref{l4}, there exists a positive constant $k_1$ such that
\begin{equation}\begin{split}\label{8}
\hat{\mathbb{E}}[e^{\lambda t}|X(t)|^2]&\leq
\hat{\mathbb{E}}|X(0)|^2+\hat{\mathbb{E}}\int_0^t e^{\lambda
s}\Big[\lambda|X(s)|^2+2X^{\tau}(s)g(X_s)\Big]ds\\&+k_1\hat{\mathbb{E}}\int_0^t
e^{\lambda s}\Big[2X^{\tau}(s)h(X_s)+|\gamma(X_s)|^2\Big]ds.
\end{split}\end{equation}
By using condition \eqref{c1} and Lemma \ref{l8} we derive
\begin{equation*}\begin{split}
X^{\tau}(t)g(X_t)&\leq (\frac{\epsilon}{2}-\lambda_1)
|X(t)|^2+\frac{1}{2\epsilon}|g(0)|^2+\lambda_2\int_{-\infty}^0|X(t+\alpha)|^2\mu_1(d\alpha).
\end{split}\end{equation*}
Similar arguments yield
\begin{equation*}\begin{split}
X^{\tau}(t)h(X_t)& \leq (\frac{\epsilon_1}{2}-\lambda_3)
|X(t)|^2+\frac{1}{2\epsilon_1}|h(0)|^2+\lambda_4\int_{-\infty}^0|X(t+\alpha)|^2\mu_2(d\alpha).\\&
\end{split}\end{equation*}
In view of condition \eqref{c3} and lemma \ref{l7} we obtain
\begin{equation*}\begin{split}
|\gamma(X_t)|^2&\leq
\frac{1}{\epsilon_2}|\gamma(0)|^2+\frac{\lambda_5}{1-\epsilon_2}\int_{-\infty}^0|X(t+\alpha)|^2\mu_3(d\alpha).
\end{split}\end{equation*}
By substituting the above obtained inequalities, \eqref{8} takes the
following form
\begin{equation}\begin{split}\label{9}
\hat{\mathbb{E}}[e^{\lambda t}|X(t)|^2]& \leq
\hat{\mathbb{E}}|X(0)|^2+\frac{1}{\lambda}\Big(\frac{1}{\epsilon}|g(0)|^2+\frac{k_1}{\epsilon_1}|h(0)|^2+\frac{k_1}{\epsilon_2}|\gamma(0)|^2\Big)(e^{\lambda
t}-1)\\& +
(\epsilon+\lambda-2\lambda_1+k_1\epsilon_1-2k_1\lambda_3)\hat{\mathbb{E}}\int_0^t
e^{\lambda s}|X(s)|^2ds\\& + 2\lambda_2\hat{\mathbb{E}}\int_0^t
e^{\lambda s}\int_{-\infty}^0|X(s+\alpha)|^2\mu_1(d\alpha)ds
\\&
+2k_1\lambda_4\hat{\mathbb{E}}\int_0^t e^{\lambda s}
\int_{-\infty}^0|X(s+\alpha)|^2\mu_2(d\alpha)ds\\&
+k_1\frac{\lambda_5}{1-\epsilon_2}\hat{\mathbb{E}}\int_0^t
e^{\lambda s}\int_{-\infty}^0|X(s+\alpha)|^2\mu_3(d\alpha)ds\\&
\end{split}\end{equation}
By result \eqref{ffn} in lemma \ref{Lf2}, we have
\begin{equation}\label{f1} \int_0^t \int_{-\infty}^0 e^{\lambda
s}|X(s+\alpha)|^2\mu_i(d\alpha)ds\leq
\frac{1}{2q-\lambda}\|\zeta\|_q^2 \mu_i^{(2q)}+\mu_i^{(2q)} \int_0^t
e^{\lambda s}|X(s)|^2ds,\end{equation} which on substituting in
\eqref{9} for $i=1,2,3$ follows
\begin{equation*}\begin{split}
&\hat{\mathbb{E}}[e^{\lambda t}|X(t)|^2] \leq
\hat{\mathbb{E}}|X(0)|^2+
\frac{2\lambda_2\mu_1^{(2q)}}{2q-\lambda}\hat{\mathbb{E}}\|\zeta\|_q^2
+
\frac{2k_1\lambda_4\mu_2^{(2q)}}{2q-\lambda}\hat{\mathbb{E}}\|\zeta\|_q^2
+\frac{k_1\lambda_5\mu_3^{(2q)}}{(2q-\lambda)(1-\epsilon_2)}\hat{\mathbb{E}}\|\zeta\|_q^2
\\&
+\frac{1}{\lambda}\Big(\frac{1}{\epsilon}|g(0)|^2+\frac{k_1}{\epsilon_1}|h(0)|^2+\frac{k_1}{\epsilon_2}|\gamma(0)|^2\Big)(e^{\lambda
t}-1)\\& -(2\lambda_1-\epsilon-\lambda-k_1\epsilon_1+2k_1\lambda_3
-2\lambda_2\mu_1^{(2q)}-2k_1\lambda_4\mu_2^{(2q)}-\frac{k_1\lambda_5}{1-\epsilon_2}\mu_3^{(2q)})\hat{\mathbb{E}}\int_0^t
e^{\lambda s}|X(s)|^2ds.
\end{split}\end{equation*}
From the assumptions we observe that $2\lambda_1>
2\lambda_2\mu_1^{(2q)}+2k_1\lambda_4\mu_2^{(2q)}+k_1\lambda_5\mu_3^{(2q)}-2k_1\lambda_3$
and $\lambda\in (0,(2\lambda_1+2k_1\lambda_3
-2\lambda_2\mu_1^{(2q)}-2k_1\lambda_4\mu_2^{(2q)}-k_1\lambda_5\mu_3^{(2q)})\wedge
2q)$. Choosing $\epsilon$,$\epsilon_1$ and $\epsilon_2$ sufficiently
small such that
\begin{equation*}2\lambda_1-\epsilon-\lambda-k_1\epsilon_1+2k_1\lambda_3
-2\lambda_2\mu_1^{(2q)}-2k_1\lambda_4\mu_2^{(2q)}-\frac{k_1\lambda_5}{1-\epsilon_2}\mu_3^{(2q)}>0,
\end{equation*}
we have
\begin{equation*}
E[|X(t)|^2] \leq K_4+K_5 e^{-\lambda t},
\end{equation*}
where
\begin{equation*}K_4=\frac{1}{\lambda}\Big(\frac{1}{\epsilon}|g(0)|^2+\frac{k_1}{\epsilon_1}|h(0)|^2+\frac{k_1}{\epsilon_2}|\gamma(0)|^2\Big)\end{equation*}
and \begin{equation*}K_5=\hat{\mathbb{E}}|X(0)|^2+
\frac{2\lambda_2\mu_1^{(2q)}}{2q-\lambda}\hat{\mathbb{E}}\|\zeta\|_q^2
+
\frac{2k_1\lambda_4\mu_2^{(2q)}}{2q-\lambda}\hat{\mathbb{E}}\|\zeta\|_q^2
+\frac{k_1\lambda_5\mu_3^{(2q)}}{(2q-\lambda)(1-\epsilon_2)}\hat{\mathbb{E}}\|\zeta\|_q^2.\end{equation*}
The proof stands completed.
\end{proof}
\begin{rem} The above theorem \ref{thm2} shows that the solution of initial value problem \eqref{1} with
 given initial data $\zeta\in C_q((-\infty,0];\R^d)$ is mean square
bounded.
\end{rem}
\begin{thm}\label{thm33} Let all the assumptions of theorem \ref{thm2} hold. Let equation
\eqref{1} has two different solutions $X(t)$ and $Y(t)$
corresponding to distinct initial data $\zeta$ and $\xi$
respectively. Then
\begin{equation}\label{f4}
\hat{\mathbb{E}}[|X(t)-Y(t)|^2] \leq
K_6\hat{\mathbb{E}}\|\zeta-\xi\|_q^2e^{-\lambda t},
\end{equation}
where
$K_6=1+\frac{1}{2q-\lambda}(2\lambda_2\mu_1^{(2q)}+2k_1\lambda_4\mu_2^{(2q)}+k_1\lambda_5\mu_3^{(2q)})$.
\end{thm}
\begin{proof} First we define $\Lambda(t)=X(t)-Y(t)$, $\hat{g}(t)=g(X_t)-g(Y_t)$, $\hat{h}(t)=h(X_t)-h(Y_t)$ and
$\hat{\gamma}(t)=\gamma(X_t)-\gamma(Y_t)$. Then applying the
G-It$\hat{o}$ formula to $e^{\lambda t}|\Lambda(t)|^2$, taking the
G-expectation on both sides, using properties of G-It\^o integral
and lemma \ref{l4}, there exists a positive constant $k_1$ such that
\begin{equation}\begin{split}\label{10}
e^{\lambda t}E|\Lambda(t)|^2&\leq
\hat{\mathbb{E}}|\zeta(0)-\xi(0)|^2+\hat{\mathbb{E}}\int_0^te^{\lambda
s}[\lambda|\Lambda(s)|^2+2\Lambda^{\tau}(s)\hat{g}(s)]ds\\&+
k_1\hat{\mathbb{E}}\int_0^te^{\lambda
s}[2\Lambda^{\tau}(s)\hat{h}(s)+|\hat{\gamma}(s)|^2]ds.
\end{split}\end{equation}
From assumption $A_{1}$, we have
\begin{equation*}\begin{split}
&\Lambda^{\tau}(t)\hat{g}(t)\leq-\lambda_1|\Lambda(t)|^2+\lambda_2\int_{-\infty}^0\Lambda(t+\alpha)\mu_1(d\alpha),\\&
\Lambda^{\tau}(t)\hat{h}(t)\leq-\lambda_3|\Lambda(t)|^2+\lambda_4\int_{-\infty}^0\Lambda(t+\alpha)\mu_2(d\alpha)\\&
\end{split}\end{equation*}
and
\begin{equation*}
|\hat{\gamma}(t)|^2\leq\lambda_5\int_{-\infty}^0\Lambda(t+\alpha)\mu_3(d\alpha).
\end{equation*}
In view of the above inequalities,  \eqref{10} takes the following
form
\begin{equation}\begin{split}\label{11}
e^{\lambda t}\hat{\mathbb{E}}|\Lambda(t)|^2& \leq
\hat{\mathbb{E}}|\zeta(0)-\xi(0)|^2
+(\lambda-2\lambda_1-2k_1\lambda_3)\hat{\mathbb{E}}\int_0^te^{\lambda
s}|\Lambda(s)|^2ds\\& +
2\lambda_2\hat{\mathbb{E}}\int_0^t\int_{-\infty}^0e^{\lambda
s}\Lambda(s+\alpha)\mu_1(d\alpha)ds +2k_1\lambda_4
\hat{\mathbb{E}}\int_0^t\int_{-\infty}^0e^{\lambda
s}\Lambda(s+\alpha)\mu_2(d\alpha)ds\\&+k_1\lambda_5\hat{\mathbb{E}}\int_0^t\int_{-\infty}^0e^{\lambda
s}\Lambda(s+\alpha)\mu_3(d\alpha)ds.
\end{split}\end{equation}
From the result \ref{ffn} of lemma \ref{Lf2} for $i=1,2,3$ we have
\begin{equation}\begin{split}\label{12} &\int_0^t \int_{-\infty}^0 e^{\lambda
s}|\Lambda(s+\alpha)|^2\mu_i(d\alpha)ds  \leq
\frac{1}{2q-\lambda}\|\zeta-\xi\|_q^2 \mu_i^{(2q)}+\mu_i^{(2q)}
\int_0^t e^{\lambda s}|\Lambda(s)|^2ds
\end{split}\end{equation}
By substituting \eqref{12} in \eqref{11}  we obtain
\begin{equation*}\begin{split}
e^{\lambda t}\hat{\mathbb{E}}|\Lambda(t)|^2&\leq
\hat{\mathbb{E}}|\zeta(0)-\xi(0)|^2 + \frac{1}{2q-\lambda}
[2\lambda_2\mu_1^{(2q)}+2k_1\lambda_4\mu_2^{(2q)}+k_1\lambda_5\mu_3^{(2q)}]\hat{\mathbb{E}}\|\zeta-\xi\|_q^2\\&
-(2\lambda_1+2k_1\lambda_3-\lambda-2\lambda_2\mu_1^{(2q)}-2k_1\lambda_4
\mu_2^{(2q)}-k_1\lambda_5\mu_3^{(2q)})\hat{\mathbb{E}}\int_0^te^{\lambda
s}|\Lambda(s)|^2ds.
\end{split}\end{equation*}
In view of the conditions $2\lambda_1>
2\lambda_2\mu_1^{(2q)}+2k_1\lambda_4\mu_2^{(2q)}+k_1\lambda_5\mu_3^{(2q)}-2k_1\lambda_3$
and
 $\lambda\in (0,(2\lambda_1+2k_1\lambda_3
-2\lambda_2\mu_1^{(2q)}-2k_1\lambda_4\mu_2^{(2q)}-k_1\lambda_5\mu_3^{(2q)})\wedge
2q)$ it follows
\begin{equation*}\begin{split}
\hat{\mathbb{E}}|\Lambda(t)|^2& \leq
[1+\frac{1}{2q-\lambda}(2\lambda_2\mu_1^{(2q)}+2k_1\lambda_4\mu_2^{(2q)}+k_1\lambda_5\mu_3^{(2q)})]\hat{\mathbb{E}}\|\zeta-\xi\|_q^2
e^{- \lambda t},
\end{split}\end{equation*}
consequently,
\begin{equation*}\begin{split}
\hat{\mathbb{E}}|X(t)-Y(t)|^2& \leq K_6
\hat{\mathbb{E}}\|\zeta-\xi\|_q^2e^{- \lambda t},
\end{split}\end{equation*}
where
$K_6=1+\frac{1}{2q-\lambda}(2\lambda_2\mu_1^{(2q)}+2k_1\lambda_4\mu_2^{(2q)}+k_1\lambda_5\mu_3^{(2q)})$.
The proof is complete.
\end{proof}
\begin{rem} The above theorem \ref{thm33} describes that two distinct solutions of the initial value problem \eqref{1} with two distinct initial
data are convergent.
\end{rem}
From theorem \ref{thm33}, we get the following stability result.
\begin{cor} Let all conditions of theorem \ref{thm33} hold. If
$g(0)=h(0)=\gamma(0)=0$, then the trivial solution of problem
\eqref{1} is mean square exponentially stable.
\end{cor}
\section{Convergence and boundedness of the solution map}
 This section
examines the asymptotic properties of the solution map $X_t$. First
we find that the solution map $X_t$ is mean square bounded. Then we
show the convergence of distinct solution maps having distinct
initial data. Let problem \eqref{1} with the given initial data
$\zeta\in C_q((-\infty,0];\R^d)$ has a unique solution $X(t)$.
\begin{thm}\label{thm3} Let assumption  $A_{1}$ holds.
Assume that $\lambda_i$, $i=1,2,..,5$ satisfy
$2\lambda_1>2\lambda_2\mu_1^{(2q)}+1+2k_1\lambda_4\mu_2^{(2q)}-2k_1\lambda_3+k_1
+2(k_1\mu_2^{(2q)}+2k_3\mu_3^{(2q)})\lambda_5$ and
$\lambda\in\Big(0,(2\lambda_1-2\lambda_2\mu_1^{(2q)}-1-2k_1\lambda_4\mu_2^{(2q)}+2k_1\lambda_3-k_1
-2(k_1\mu_2^{(2q)}+2k_3\mu_3^{(2q)})\lambda_5)\wedge 2q\Big)$. For
any initial data $\zeta\in C_q((-\infty,0];\R^d)$, we then have
\begin{equation*}
\hat{\mathbb{E}}\|X_t\|_q^2 \leq K_7+K_8e^{-\lambda t},
\end{equation*}
where
$K_7=\frac{2}{\lambda}\Big(|g(0)|^2+k_1|h(0)|^2+2(k_1+2k_3)|\gamma(0)|^2\Big)$
and $K_8=3+\frac{4}{2q-\lambda}[ \lambda_2\mu_1^{(2q)}+
k_1(\lambda_4+\lambda_5)\mu_2^{(2q)}+2k_3\lambda_5\mu_3^{(2q)}]$.
\end{thm}
\begin{proof}
Applying the G-It$\hat{o}$ formula to $e^{\lambda t}|X(t)|^2$ and
taking the G-expectation on both sides, we have
\begin{equation}\begin{split}\label{13}
\hat{\mathbb{E}}\Big[\sup_{0<s\leq t}e^{\lambda s}|X(s)|^2\Big]&\leq
\hat{\mathbb{E}}|\zeta(0)|^2+\hat{\mathbb{E}}\Big[\sup_{0<s\leq
t}\int_0^t e^{\lambda
s}\Big(\lambda|X(s)|^2+2X^{\tau}(s)g(X_s)\Big)ds\Big]\\&+\hat{\mathbb{E}}\Big[\sup_{0<s\leq
t}\int_0^t e^{\lambda
s}\Big(2X^{\tau}(s)h(X_s)+|\gamma(X_s)|^2\Big)d \langle B, B \rangle
(s)\Big]\\& +2\hat{\mathbb{E}}\Big[\sup_{0<s\leq t}\int_0^t
e^{\lambda s}X^{\tau}(s)\gamma(X_s)dB(s)\Big].
\end{split}\end{equation}
By straightforward calculations, using \eqref{4} and then \eqref{f1}
we obtain
\begin{equation}\begin{split}\label{14}&\hat{\mathbb{E}}\Big[\sup_{0<s\leq t}\int_0^t e^{\lambda
s}\Big(\lambda|X(s)|^2+2X^{\tau}(s)g(X_s)\Big)ds\Big]\\&\leq\frac{1}{\lambda}|g(0)|^2(e^{\lambda
t}-1)+\frac{2\lambda_2}{2q-\lambda}\hat{\mathbb{E}}\|\zeta\|_q^2
\mu_1^{(2q)}+(\lambda-2\lambda_1+2\lambda_2\mu_1^{(2q)}+1)\hat{\mathbb{E}}\int_0^t
e^{\lambda s}|X(s)|^2ds.
\end{split}\end{equation}
By using \eqref{5}, \eqref{6}, \eqref{f1} and lemma \ref{l4}, there
exists a positive constant $k_1$ such that
\begin{equation*}\begin{split}
&\hat{\mathbb{E}}\Big[\sup_{0<s\leq t}\int_0^t e^{\lambda
s}\Big(2X^{\tau}(s)h(X_s)+|\gamma(X_s)|^2\Big)d \langle B, B \rangle
(s)\Big]\\& \leq
\frac{1}{\lambda}k_1(|h(0)|^2+2|\gamma(0)|^2)(e^{\lambda t}-1)
-k_1(2\lambda_3-1)\hat{\mathbb{E}}\int_0^t e^{\lambda
s}|X(s)|^2ds\\& +2k_1\lambda_4\hat{\mathbb{E}}\int_0^t
\int_{-\infty}^0e^{\lambda s}|X(s+\alpha)|^2\mu_2(d\alpha)ds\\&
+2k_1\lambda_5\hat{\mathbb{E}}\int_0^t \int_{-\infty}^0e^{\lambda
s}|X(s+\alpha)|^2\mu_3(d\alpha)ds\\& \leq
\frac{1}{\lambda}k_1(|h(0)|^2+2|\gamma(0)|^2)(e^{\lambda t}-1)
-k_1(2\lambda_3-1)\hat{\mathbb{E}}\int_0^t e^{\lambda
s}|X(s)|^2ds\\&
+\frac{2k_1\lambda_4}{2q-\lambda}\hat{\mathbb{E}}\|\zeta\|_q^2
\mu_2^{(2q)}+2k_1\lambda_4\mu_2^{(2q)} \hat{\mathbb{E}}\int_0^t
e^{\lambda s}|X(s)|^2ds\\&
+\frac{2k_1\lambda_5}{2q-\lambda}\hat{\mathbb{E}}\|\zeta\|_q^2
\mu_2^{(2q)}+2k_1\lambda_5\mu_2^{(2q)} \hat{\mathbb{E}}\int_0^t
e^{\lambda s}|X(s)|^2ds,
\end{split}\end{equation*}
simplification yields
\begin{equation}\begin{split}\label{15}
&\hat{\mathbb{E}}\Big[\sup_{0<s\leq t}\int_0^t e^{\lambda
s}\Big(2X^{\tau}(s)h(X_s)+|\gamma(X_s)|^2\Big)d \langle B, B \rangle
(s)\Big]\\& \leq
\frac{1}{\lambda}k_1(|h(0)|^2+2|\gamma(0)|^2)(e^{\lambda t}-1)
+\frac{2}{2q-\lambda}(k_1\lambda_4\mu_2^{(2q)}+k_1\lambda_5\mu_2^{(2q)})\hat{\mathbb{E}}\|\zeta\|_q^2\\&
+(2k_1\lambda_4\mu_2^{(2q)}+2k_1\lambda_5\mu_2^{(2q)}-2k_1\lambda_3+k_1)\hat{\mathbb{E}}\int_0^t
e^{\lambda s}|X(s)|^2ds.
\end{split}\end{equation}
By utilizing \eqref{6}, the inequality
$a_1a_2\leq\frac{1}{2}\sum_{i=1}^2a_i$ and lemma \ref{l3}, there
exists a positive constant $k_2$ such that
\begin{equation*}\begin{split}
2\hat{\mathbb{E}}\Big[\sup_{0<s\leq t}\int_0^t e^{\lambda
s}X^{\tau}(s)\gamma(X_s)dB(t)\Big]& \leq
2k_2\hat{\mathbb{E}}\Big[\int_0^t e^{\lambda s}|X(s)|^2e^{\lambda
s}|\gamma(X_s)|^2ds\Big]^\frac{1}{2}\\& \leq
\frac{1}{2}E\Big[\sup_{0<s\leq t}e^{\lambda
s}|X(s)|^2\Big]+2k^2_2E\int_0^t e^{\lambda s}|\gamma(X_s)|^2ds\\&
\leq\frac{1}{2}\hat{\mathbb{E}}\Big[\sup_{0<s\leq t}e^{\lambda
s}|X(s)|^2\Big] +4k^2_2\frac{1}{\lambda}|\gamma(0)|^2(e^{\lambda
t}-1)\\& +4k^2_2\lambda_5\hat{\mathbb{E}}\int_0^t
\int_{-\infty}^0e^{\lambda s}|X(s+\alpha)|^2\mu_3(d\alpha),
\end{split}\end{equation*}
by using lemma \ref{Lf2}, we get
\begin{equation}\begin{split}\label{16}
2\hat{\mathbb{E}}\Big[\sup_{0<s\leq t}\int_0^t e^{\lambda
s}X^{\tau}(s)\gamma(X_s)dB(t)\Big]& \leq
\frac{1}{2}\hat{\mathbb{E}}\Big[\sup_{0<s\leq t}e^{\lambda
s}|X(s)|^2\Big] +4k_3\frac{1}{\lambda}|\gamma(0)|^2(e^{\lambda
t}-1)\\&
+\frac{4k_3\lambda_5}{2q-\lambda}\mu_3^{(2q)}\hat{\mathbb{E}}\|\zeta\|_q^2
+4k_3\lambda_5\mu_3^{(2q)} \hat{\mathbb{E}}\int_0^t e^{\lambda
s}|X(s)|^2ds,
\end{split}\end{equation}
where $k_3=k^2_2$. Noticing that $\zeta(0)\leq
\sup_{-\infty<\alpha\leq 0}e^{q\alpha}|\zeta(\alpha)|=\|\zeta\|_q$
and substituting \eqref{14}, \eqref{15} and \eqref{16} in
\eqref{13}, it follows
\begin{equation*}\begin{split}
\hat{\mathbb{E}}\Big[\sup_{0<s\leq t}e^{\lambda s}|X(s)|^2\Big]&
\leq
\frac{2}{\lambda}\Big(|g(0)|^2+k_1|h(0)|^2+2(k_1+2k_3)|\gamma(0)|^2\Big)(e^{\lambda
t}-1)\\& +\frac{2}{2q-\lambda}\Big(
2q-\lambda+2\lambda_2\mu_1^{(2q)}+
2k_1(\lambda_4+\lambda_5)\mu_2^{(2q)}+4k_3\lambda_5\mu_3^{(2q)}\Big)\hat{\mathbb{E}}\|\zeta\|_q^2\\&
-2\Big(2\lambda_1-2\lambda_2\mu_1^{(2q)}-1-\lambda\\&-2k_1\lambda_4\mu_2^{(2q)}+2k_1\lambda_3-k_1
-2(k_1\mu_2^{(2q)}+2k_3\mu_3^{(2q)})\lambda_5\Big)\hat{\mathbb{E}}\int_0^t
e^{\lambda s}|X(s)|^2ds.
\end{split}\end{equation*}
By using the assumptions
$2\lambda_1>2\lambda_2\mu_1^{(2q)}+1+2k_1\lambda_4\mu_2^{(2q)}-2k_1\lambda_3+k_1
+2(k_1\mu_2^{(2q)}+2k_3\mu_3^{(2q)})\lambda_5$ and
$\lambda\in\Big(0,(2\lambda_1-2\lambda_2\mu_1^{(2q)}-1-2k_1\lambda_4\mu_2^{(2q)}+2k_1\lambda_3-k_1
-2(k_1\mu_2^{(2q)}+2k_3\mu_3^{(2q)})\lambda_5)\wedge 2q\Big)$, we
get
\begin{equation}\begin{split}\label{17}
\hat{\mathbb{E}}\Big[\sup_{0<s\leq t}e^{\lambda s}|X(s)|^2\Big]&
\leq
\frac{2}{\lambda}\Big(|g(0)|^2+k_1|h(0)|^2+2(k_1+2k_3)|\gamma(0)|^2\Big)(e^{\lambda
t}-1)\\& +\frac{2}{2q-\lambda}\Big(
2q-\lambda+2\lambda_2\mu_1^{(2q)}+
2k_1(\lambda_4+\lambda_5)\mu_2^{(2q)}+4k_3\lambda_5\mu_3^{(2q)}\Big)\hat{\mathbb{E}}\|\zeta\|_q^2.
\end{split}\end{equation}
From lemma \eqref{Lf3},  we have
\begin{equation*}\begin{split}\label{}
\hat{\mathbb{E}}\|X_t\|^2_q& \leq e^{-\lambda t}
\hat{\mathbb{E}}\|\zeta\|_q^2+e^{-\lambda
t}\hat{\mathbb{E}}\Big[\sup_{0<s\leq t}e^{\lambda s}|X(s)|^2\Big].
\end{split}\end{equation*}
Substituting \eqref{17} in the above inequality yields
\begin{equation*}\begin{split}\label{}
\hat{\mathbb{E}}\|X_t\|& \leq e^{-\lambda t}
\hat{\mathbb{E}}\|\zeta\|_q^2+\frac{2}{\lambda}\Big(|g(0)|^2+k_1|h(0)|^2+2(k_1+2k_3)|\gamma(0)|^2\Big)(1-e^{-\lambda
t})\\& +\frac{2}{2q-\lambda}\Big( 2q-\lambda+2\lambda_2\mu_1^{(2q)}+
2k_1(\lambda_4+\lambda_5)\mu_2^{(2q)}+4k_3\lambda_5\mu_3^{(2q)}\Big)\hat{\mathbb{E}}\|\zeta\|_q^2]e^{-\lambda
t}\\& \leq K_7+K_8e^{-\lambda t}
\end{split}\end{equation*}
where
$K_7=\frac{2}{\lambda}\Big(|g(0)|^2+k_1|h(0)|^2+2(k_1+2k_3)|\gamma(0)|^2\Big)$
and $K_8=3+\frac{4}{2q-\lambda}\Big( \lambda_2\mu_1^{(2q)}+
k_1(\lambda_4+\lambda_5)\mu_2^{(2q)}+2k_3\lambda_5\mu_3^{(2q)}\Big)$.
The proof is complete.
\end{proof}
\begin{rem} Theorem \ref{thm3} states that the solution map $X_t$ of the initial value problem \ref{1} with given initial data $\zeta\in C_q((-\infty,0];\R^d)$ is mean square
bounded.
\end{rem}
\begin{thm}\label{thm4} Let assumptions $A_{1}$ holds. Assume that
$\lambda_i$, $i=1,2,..,5$ satisfy $2\lambda_1>2\lambda_2\mu_1^{(2q)}
-2k_1\lambda_3+2k_1\lambda_4\mu_2^{(2q)}+(k_1+2k_3)\mu_3^{(2q)}\lambda_5$
and $\lambda\in(0,(2\lambda_1-2\lambda_2\mu_1^{(2q)}
+2k_1\lambda_3-2k_1\lambda_4\mu_2^{(2q)}-(k_1+2k_3)\mu_3^{(2q)}\lambda_5)\wedge
2q)$. Then for distinct initial data $\zeta,\xi\in C_q$ the
respective solution maps $X_t^{\zeta}$ and $Y_t^{\xi}$ satisfy
\begin{equation*} \hat{\mathbb{E}}\|X_t^{\zeta}-Y_t^{\xi}\|_q^2\leq
K_9\hat{\mathbb{E}}\|\zeta-\xi\|_q^2e^{-\hat{\lambda} t},
\end{equation*}
where $K_9=1+\frac{4}{2q-\lambda}[\lambda_2\mu_1^{(2q)}
+k_1(2\lambda_4\mu_2^{(2q)}+\lambda_5\mu_3^{(2q)})+k_3\lambda_5\mu_3^{(2q)}]$.
\end{thm}
\begin{proof} Keeping in mind the definitions of $\Lambda(t)$,
$\hat{g}(t)$, $\hat{h}(t)$ and $\hat{\gamma}(t)$, we apply the
G-It$\hat{o}$ formula to $e^{\lambda t}|\Lambda(t)|^2$ and then
taking G-expectation on both sides to obtain
\begin{equation}\begin{split}\label{18}
\hat{\mathbb{E}}[\sup_{0<s\leq t}e^{\lambda t}|\Lambda(s)|^2]&\leq
\hat{\mathbb{E}}|\zeta(0)-\xi(0)|^2+\hat{\mathbb{E}}[\sup_{0<s\leq
t}\int_0^te^{\lambda
s}[\lambda|\Lambda(s)|^2+2\Lambda^{\tau}(s)\hat{g}(s)]ds]\\&+
\hat{\mathbb{E}}[\sup_{0<s\leq t}\int_0^te^{\lambda
s}[2\Lambda^{\tau}(s)\hat{h}(s)+|\hat{\gamma}(s)|^2]d \langle B, B
\rangle (s)]\\&+2\hat{\mathbb{E}}[\sup_{0<s\leq t}\int_0^t
e^{\lambda s}\Lambda^{\tau}(s)\gamma(X_s)dB(s)].
\end{split}\end{equation}
By using \eqref{c1} and then lemma \ref{Lf2} we have
\begin{equation*}\begin{split}
&\hat{\mathbb{E}}[\sup_{0<s\leq t}\int_0^te^{\lambda
s}[\lambda|\Lambda(s)|^2+2\Lambda^{\tau}(s)\hat{g}(s)]ds]\\&
\leq\frac{2\lambda_2\mu_1^{(2q)}}{2q-\lambda}\hat{\mathbb{E}}\|\zeta-\xi\|_q^2
- (2\lambda_1-\lambda-2\lambda_2\mu_1^{(2q)})
\hat{\mathbb{E}}\int_0^te^{\lambda s}|\Lambda(s)|^2ds.
\end{split}\end{equation*}
By using \eqref{c2}, \eqref{c3}, lemma \ref{Lf2} and lemma \ref{l4},
there exists a positive constant $k_1$ such that
\begin{equation*}\begin{split}
 &\hat{\mathbb{E}}[\sup_{0<s\leq
t}\int_0^te^{\lambda
s}[2\Lambda^{\tau}(s)\hat{h}(s)+|\hat{\gamma}(s)|^2]d \langle B, B
\rangle (s)]\\& \leq
-2k_1\lambda_3\hat{\mathbb{E}}\int_0^te^{\lambda s} |\Lambda(s)|^2ds
+2k_1\lambda_4\hat{\mathbb{E}}\int_0^t\int_{-\infty}^0e^{\lambda
s}|\Lambda(s+\alpha)|^2\mu_2(d\alpha)ds\\&
+k_1\lambda_5\hat{\mathbb{E}}\int_0^t\int_{-\infty}^0e^{\lambda
s}|\Lambda(s+\alpha)|^2\mu_3(d\alpha)ds\\&
\leq\frac{1}{2q-\lambda}k_1(2\lambda_4\mu_2^{(2q)}+\lambda_5\mu_3^{(2q)})\hat{\mathbb{E}}\|\zeta-\xi\|_q^2
-(2k_1\lambda_3-2k_1\lambda_4\mu_2^{(2q)}-\lambda_5k_1\mu_3^{(2q)} )
\hat{\mathbb{E}}\int_0^te^{\lambda s} |\Lambda(s)|^2ds.
\end{split}\end{equation*}
By utilizing \eqref{6}, the inequality
$a_1a_2\leq\frac{1}{2}\sum_{i=1}^2a_i$ and lemma \ref{l3}, there
exists a positive constant $k_2$ such that
\begin{equation*}\begin{split}
&2\hat{\mathbb{E}}[\sup_{0<s\leq t}\int_0^t e^{\lambda
s}\Lambda^{\tau}(s)\hat{\gamma}(X_s)dB(s)] \leq
2k_2\hat{\mathbb{E}}(\int_0^t e^{\lambda s}|\Lambda(s)|^2e^{\lambda
s}|\hat{\gamma}(X_s)|^2ds)^\frac{1}{2}\\& \leq
\hat{\mathbb{E}}(\sup_{0<s\leq t}e^{\lambda
s}|F(s)|^2)^\frac{1}{2}(4k^2_2\int_0^t e^{\lambda
s}|\hat{\gamma}(X_s)|^2ds)^\frac{1}{2}\\& \leq
\frac{1}{2}\hat{\mathbb{E}}(\sup_{0<s\leq t}e^{\lambda
s}|\Lambda(s)|^2)+2k^2_2\lambda_5\hat{\mathbb{E}}\int_0^t
\int_{-\infty}^0e^{\lambda s}|\Lambda(s+\alpha)|^2\mu_3(d\alpha)ds,
\end{split}\end{equation*}
by using lemma \ref{Lf2}, we get
\begin{equation}\begin{split}\label{16}
&2\hat{\mathbb{E}}[\sup_{0<s\leq t}\int_0^t e^{\lambda
s}\Lambda^{\tau}(s)\hat{\gamma}(X_s)dB(s)]\\& \leq
\frac{1}{2}\hat{\mathbb{E}}(\sup_{0<s\leq t}e^{\lambda
s}|\Lambda(s)|^2)
+\frac{2k_3\lambda_5}{2q-\lambda}\mu_3^{(2q)}\hat{\mathbb{E}}\|\zeta-\xi\|_q^2
+2k_3\lambda_5\mu_3^{(2q)} \hat{\mathbb{E}}\int_0^t e^{\lambda
s}|X(s)|^2ds,
\end{split}\end{equation}
where $k_3=k^2_2$. Substituting all the above derived inequalities
in \eqref{18} we derive
\begin{equation*}\begin{split}
&\hat{\mathbb{E}}[\sup_{0<s\leq t}e^{\lambda t}|\Lambda(t)|^2] \leq
\frac{1}{2q-\lambda}[2\lambda_2\mu_1^{(2q)}
+k_1(2\lambda_4\mu_2^{(2q)}+\lambda_5\mu_3^{(2q)})+2k_3\lambda_5\mu_3^{(2q)}]\hat{\mathbb{E}}\|\zeta-\xi\|_q^2\\&
- [2\lambda_1-\lambda-2\lambda_2\mu_1^{(2q)}
+2k_1\lambda_3-2k_1\lambda_4\mu_2^{(2q)}-\lambda_5k_1\mu_3^{(2q)}-2k_3\lambda_5\mu_3^{(2q)}
] \hat{\mathbb{E}}\int_0^te^{\lambda s}|\Lambda(s)|^2ds\\&
+\frac{1}{2}\hat{\mathbb{E}}(\sup_{0<s\leq t}e^{\lambda
s}|\Lambda(s)|^2),
\end{split}\end{equation*}
simplification yields
\begin{equation*}\begin{split}
&\hat{\mathbb{E}}[\sup_{0<s\leq t}e^{\lambda t}|\Lambda(t)|^2]\leq
\frac{2}{2q-\lambda}[2\lambda_2\mu_1^{(2q)}
+2k_1(2\lambda_4\mu_2^{(2q)}+\lambda_5\mu_3^{(2q)})+2k_3\lambda_5\mu_3^{(2q)}]\hat{\mathbb{E}}\|\zeta-\xi\|_q^2\\&
- 2[2\lambda_1-\lambda-2\lambda_2\mu_1^{(2q)}
+2k_1\lambda_3-2k_1\lambda_4\mu_2^{(2q)}-(k_1+2k_3)\mu_3^{(2q)}\lambda_5
] \hat{\mathbb{E}}\int_0^te^{\lambda s}|\Lambda(s)|^2ds.
\end{split}\end{equation*}
By using the assumptions $2\lambda_1>2\lambda_2\mu_1^{(2q)}
-2k_1\lambda_3+2k_1\lambda_4\mu_2^{(2q)}+(k_1+2k_3)\mu_3^{(2q)}\lambda_5$
and $\lambda\in(0,(2\lambda_1-2\lambda_2\mu_1^{(2q)}
+2k_1\lambda_3-2k_1\lambda_4\mu_2^{(2q)}-(k_1+2k_3)\mu_3^{(2q)}\lambda_5)\wedge
2q)$, we get
\begin{equation}\begin{split}\label{19}
\hat{\mathbb{E}}[\sup_{0<s\leq t}e^{\lambda s}|\Lambda(s)|^2]&\leq
\frac{2}{2q-\lambda}[2\lambda_2\mu_1^{(2q)}
+2k_1(2\lambda_4\mu_2^{(2q)}+\lambda_5\mu_3^{(2q)})+2k_3\lambda_5\mu_3^{(2q)}]\hat{\mathbb{E}}\|\zeta-\xi\|_q^2\\&
\end{split}\end{equation}
In view of lemma \ref{Lf3}, we have
\begin{equation*}\begin{split}
\hat{\mathbb{E}}\|X_t^{\zeta}-Y_t^{\xi}\|_q^2&\leq e^{-\lambda
t}\hat{\mathbb{E}}\|\zeta-\xi\|_q^2 +e^{-\lambda
t}\hat{\mathbb{E}}(\sup_{0<s\leq t}e^{\lambda s}|\Lambda(s)|^2),
\end{split}\end{equation*}
on substituting \eqref{19} in the above inequality, we derive
\begin{equation*}\begin{split}
&\hat{\mathbb{E}}\|X_t^{\zeta}-Y_t^{\xi}\|_q^2\\&\leq e^{-\lambda
t}\hat{\mathbb{E}}\|\zeta-\xi\|_q^2
+\frac{4}{2q-\lambda}[\lambda_2\mu_1^{(2q)}
+k_1(2\lambda_4\mu_2^{(2q)}+\lambda_5\mu_3^{(2q)})+k_3\lambda_5\mu_3^{(2q)}]\hat{\mathbb{E}}\|\zeta-\xi\|_q^2e^{-\lambda
t}\\&=K_9\hat{\mathbb{E}}\|\zeta-\xi\|_q^2e^{-\lambda t},
\end{split}\end{equation*}
where $K_9=1+\frac{4}{2q-\lambda}[\lambda_2\mu_1^{(2q)}
+k_1(2\lambda_4\mu_2^{(2q)}+\lambda_5\mu_3^{(2q)})+k_3\lambda_5\mu_3^{(2q)}]$.
The proof stands completed.
\end{proof}
\begin{rem} Theorem \ref{thm4} indicates that two distinct solution
maps $X_t$ and $Y_t$  from the respective distinct initial data
$\zeta\in C_q((-\infty,0];\R^d)$ and $\xi\in C_q((-\infty,0];\R^d)$
are convergent.
\end{rem}
\section{The exponential estimate} To establish the exponential estimate, assume that equation \eqref{1} with initial
data $\zeta\in C_q((-\infty,0];\R^d)$ has a unique solution $X(t)$
on $t\in[0,\infty)$. First, we derive the $L^2_G$ and then the
exponential estimates as follows.
\begin{thm} Let $E\|\zeta\|^2_q<\infty$ and assumption $A_1$ hold.
Then for all $t\geq 0$,
 \begin{equation*}
\hat{\mathbb{E}}\Big[\sup_{-\infty< s\leq t}|X(t)|^2\Big]\leq
[\hat{\mathbb{E}}\|\zeta\|^2_q+L_1]e^{L_2t},\end{equation*} where
$L_1=\hat{K} +
\frac{2}{q}[q+\lambda_2\mu_1^{(2q)}+k_1(\lambda_5\mu_3^{(2q)}+\mu_2^{(2q)})+2k_3\lambda_5\mu_3^{(2q)}]\hat{\mathbb{E}}\|\zeta\|^2_q$,
$\hat{K}=2[|g(0)|^2+k_1(|h(0)|^2+2|\gamma(0)|^2)+4k_3|\gamma(0)|^2]T$
and
$L_2=2[2\lambda_2-2\lambda_1+1+k_1(2\lambda_5-2\lambda_3+3)+4k_3\lambda_5]$.
\end{thm}
\begin{proof} Applying the G-It$\hat{o}$ formula to $|X(t)|^2$ and taking the
G-expectation on both sides
\begin{equation}\begin{split}\label{20}
\hat{\mathbb{E}}\Big[\sup_{0\leq s\leq t}|X(t)|^2\Big]&\leq
\hat{\mathbb{E}}|X(0)|^2+2\hat{\mathbb{E}}\Big[\sup_{0\leq s\leq
t}\int_0^{t}X^\tau(s)
g(X_s)ds\Big]\\&+\hat{\mathbb{E}}\Big[\sup_{0\leq s\leq
t}\int_0^{t}(2X^\tau(s) h(X_s)+|\gamma(X_s)|^2) d \langle B, B
\rangle (s)\Big]\\&+2\hat{\mathbb{E}}\Big[\sup_{0\leq s\leq
t}\int_0^t X^{\tau}(s)\gamma(X_s)dB(s)\Big]
\end{split}\end{equation}
By using \eqref{4} and then \eqref{f1} we obtain
\begin{equation*}\begin{split}&2\hat{\mathbb{E}}\Big[\sup_{0<s\leq t}\int_0^t X^{\tau}(s)g(X_s)ds\Big]
\leq |g(0)|^2T+
\frac{\lambda_2}{q}\hat{\mathbb{E}}\|\zeta\|^2_q\mu_1^{(2q)}
+(2\lambda_2-2\lambda_1+1)\hat{\mathbb{E}}\int_{0}^{t}|X(s)|^2ds.
\end{split}\end{equation*}
By using \eqref{5}, \eqref{6}, \eqref{f1} and lemma \ref{l4}, there
exists a positive constant $k_1$ such that
\begin{equation*}\begin{split}
&\hat{\mathbb{E}}\Big[\sup_{0\leq s\leq t}\int_0^{t}(2X^\tau(s)
h(X_s)+|\gamma(X_s)|^2) d \langle B, B \rangle (s)\Big]\leq
k_1\hat{\mathbb{E}}\Big[\int_0^{t}(2X^\tau(s)
h(X_s)+|\gamma(X_s)|^2)\Big] ds
\\&\leq k_1[|h(0)|^2+2|\gamma(0)|^2]T+
k_1\frac{1}{q}(\lambda_5\mu_3^{(2q)}+\mu_2^{(2q)})\hat{\mathbb{E}}\|\zeta\|^2_q\\&
+k_1(2\lambda_5-2\lambda_3+3)\hat{\mathbb{E}}\int_{0}^{t}|X(s)|^2ds.
\end{split}\end{equation*}
By utilizing \eqref{6}, the inequality
$a_1a_2\leq\frac{1}{2}\sum_{i=1}^2a_i$ and lemma \ref{l3},
straightforward calculations give
\begin{equation*}\begin{split}
&2\hat{\mathbb{E}}\Big[\sup_{0<s\leq t}\int_0^t
X^{\tau}(s)\gamma(X_s)dB(t)\Big] \leq
2k_2\hat{\mathbb{E}}\Big[\int_0^t
|X^{\tau}(s)\gamma(X_s)|^2ds\Big]^\frac{1}{2}\\&
\leq\frac{1}{2}\hat{\mathbb{E}}\Big[\sup_{0<s\leq t}|X(s)|^2\Big]
+4k^2_2|\gamma(0)|^2T +4k^2_2\lambda_5\hat{\mathbb{E}}\int_0^t
\int_{-\infty}^0|X(s+\alpha)|^2\mu_3(d\alpha)ds,
\end{split}\end{equation*}
by using lemma \ref{Lf2}, we get
\begin{equation}\begin{split}\label{16}
&2E\Big[\sup_{0<s\leq t}\int_0^t X^{\tau}(s)\gamma(X_s)dB(s)\Big]\\&
\leq \frac{1}{2}\hat{\mathbb{E}}\Big[\sup_{0<s\leq t}|X(s)|^2\Big]
+4k_3|\gamma(0)|^2T
+\frac{4k_3\lambda_5}{2q}\mu_3^{(2q)}\hat{\mathbb{E}}\|\zeta\|_q^2
+4k_3\lambda_5 \hat{\mathbb{E}}\int_0^t |X(s)|^2ds,
\end{split}\end{equation}
where $k_3=k^2_2$. Substituting the above inequalities in \eqref{20}
and then by straightforward calculations, we derive
\begin{equation}\begin{split}\label{fn1}
\hat{\mathbb{E}}[\sup_{0\leq s\leq t}|X(t)|^2]&\leq L_1
+L_2\int_{0}^{t}\hat{\mathbb{E}}\Big[\sup_{0\leq s\leq
t}|X(s)|^2\Big]ds,
\end{split}\end{equation}
where $L_1=\hat{K} +
\frac{2}{q}[q+\lambda_2\mu_1^{(2q)}+k_1(\lambda_5\mu_3^{(2q)}+\mu_2^{(2q)})+2k_3\lambda_5\mu_3^{(2q)}]E\|\zeta\|^2_q$,
$\hat{K}=2\Big[|g(0)|^2+k_1(|h(0)|^2+2|\gamma(0)|^2)+4k_3|\gamma(0)|^2\Big]T$
and
$L_2=2[2\lambda_2-2\lambda_1+1+k_1(2\lambda_5-2\lambda_3+3)+4k_3\lambda_5]$.
By noticing that
\begin{equation*}
\hat{\mathbb{E}}[\sup_{-\infty< s\leq t}|X(s)|^2\Big]\leq
\hat{\mathbb{E}}\|\zeta\|_q^2+ \hat{\mathbb{E}}\Big[\sup_{0\leq
s\leq t}|X(s)|^2],
\end{equation*}
it follows
\begin{equation*}\begin{split}
\hat{\mathbb{E}}\Big[\sup_{-\infty< s\leq t}|X(s)|^2\Big]& \leq
\hat{\mathbb{E}}\|\zeta\|_q^2+ L_1
+L_2\int_{0}^{t}\hat{\mathbb{E}}\Big[\sup_{0\leq s\leq
t}|X(s)|^2\Big]ds \\&\leq \hat{\mathbb{E}}\|\zeta\|_q^2+ L_1
+L_2\int_{0}^{t}\hat{\mathbb{E}}\Big[\sup_{-\infty< s\leq
t}|X(s)|^2\Big]ds.
\end{split}\end{equation*}
By applying the Grownwall inequality, we get the desired expression.
\end{proof}
\begin{thm} Let $\hat{\mathbb{E}}\|\zeta\|^2_q<\infty$ and assumption $A_1$ hold.
Then for all $t\geq 0$,
\begin{equation*}
\lim_{t\rightarrow\infty}\sup \frac{1}{t}log|X(t)|\leq M,
\end{equation*}
where
$M=2\lambda_2-2\lambda_1+1+k_1(2\lambda_5-2\lambda_3+3)+4k_3\lambda_5$.
\end{thm}
\begin{proof} Applying the Grownwall inequality from \eqref{fn1}, it
follows
\begin{equation}\label{21}
\hat{\mathbb{E}}\Big[\sup_{0\leq s\leq t}|X(s)|^2\Big]\leq
L_1e^{L_2t},\end{equation} where $L_1=\hat{K} +
\frac{2}{q}[q+\lambda_2\mu_1^{(2q)}+k_1(\lambda_5\mu_3^{(2q)}+\mu_2^{(2q)})+2k_3\lambda_5\mu_3^{(2q)}]E\|\zeta\|^2_q$,
$\hat{K}=2\Big[|g(0)|^2+k_1(|h(0)|^2+2|\gamma(0)|^2)+4k_3|\gamma(0)|^2\Big]T$
and
$L_2=2[2\lambda_2-2\lambda_1+1+k_1(2\lambda_5-2\lambda_3+3)+4k_3\lambda_5]$.
By virtue of the above result \eqref{21}, for each $m=1,2,3,...,$ we
have
\begin{equation*}
\hat{\mathbb{E}}\Big[\sup_{m-1\leq t\leq m}|X(t)|^2\Big]\leq
L_1e^{L_2m}.\end{equation*} For any $\epsilon>0$, by using lemma
\ref{l2} we get
\begin{equation*}\begin{split}
\hat{C}\Big\{w:\sup_{m-1\leq t\leq
m}|X(t)|^2>e^{(L_2+\epsilon)m}\Big\}&\leq
\frac{\hat{\mathbb{E}}\Big[\sup_{m-1\leq t\leq
m}|X(t)|^2\Big]}{e^{(L_2+\epsilon)m}}\\&
\leq\frac{L_1e^{L_2m}}{e^{(L_2+\epsilon)m}}\\& =L_1e^{-\epsilon m}.
\end{split}\end{equation*}
But for almost all $w\in\Omega$, the Borel-Cantelli lemma yields
that there exists a random integer $m_0=m_0(w)$ so that
\begin{equation*}
\sup_{m-1\leq t\leq m}|X(t)|^2\leq e^{(L_2+\epsilon)m},\,\,
\textit{whenever}\,\,m\geq m_0,\end{equation*} which implies
\begin{equation*}\begin{split}
\lim_{t\rightarrow\infty}\sup \frac{1}{t}log|X(t)|&\leq
\frac{L_2+\epsilon}{2}\\&=
2\lambda_2-2\lambda_1+1+k_1(2\lambda_5-2\lambda_3+3)+4k_3\lambda_5+\frac{\epsilon}{2},
\end{split}\end{equation*}
but $\epsilon$ is arbitrary and the above result reduces to
\begin{equation*}
\lim_{t\rightarrow\infty}\sup \frac{1}{t}log|X(t)|\leq M,
\end{equation*}
where
$M=2\lambda_2-2\lambda_1+1+k_1(2\lambda_5-2\lambda_3+3)+4k_3\lambda_5$.
The proof is complete.
\end{proof}
\begin{rem}
The above lemma states that the second moment of Lyapunov exponent
\cite{f1,k} $\lim_{t\rightarrow \infty}\sup\frac{1}{t}log|X(t)|$ is
bounded with upper bound M.
\end{rem}
\section{Acknowledgement} This work is supported by the Commonwealth
Scholarship Commission in the United Kingdom with project ID Number:
PKRF-2017-429. My heartfelt gratitude is due to Professor Chenggui
Yuan who not only guided me to embark on this endeavor but also
enriched my work through his insightful input throughout this
research.

\end{document}